\documentclass[12pt]{amsart}

\usepackage{amsmath}
\usepackage{amssymb}
\usepackage{amscd}
\def\NZQ{\mathbb}               % the font for N,Z,Q,R,C
\def\NN{{\NZQ N}}
\def\QQ{{\NZQ Q}}
\def\ZZ{{\NZQ Z}}
\def\RR{{\NZQ R}}

\def\PP{{\NZQ P}}

\newtheorem{Theorem}{Theorem}[section]
\newtheorem{Lemma}[Theorem]{Lemma}
\newtheorem{Corollary}[Theorem]{Corollary}
\newtheorem{Proposition}[Theorem]{Proposition}
\newtheorem{Remark}[Theorem]{Remark}

\newtheorem{Example}[Theorem]{Example}

\newtheorem{Definition}[Theorem]{Definition}

\let\epsilon\varepsilon
\let\phi=\varphi
\let\kappa=\varkappa

\begin{document}

\title{The Rees algebra and analytic spread of a divisorial filtration}
\author{Steven Dale Cutkosky}

\thanks{The first author was partially supported by NSF grant DMS-2054394.}

\address{Steven Dale Cutkosky, Department of Mathematics,
University of Missouri, Columbia, MO 65211, USA}
\email{cutkoskys@missouri.edu}

\begin{abstract} In this paper we investigate some properties of Rees algebras of divisorial filtrations and their analytic spread.
A classical theorem of McAdam shows that the analytic spread of an ideal $I$ in a formally equidimensional local ring is equal to the dimension of the ring if and only if the maximal ideal is an associated prime of $R/\overline{I^n}$ for some $n$. We show in Theorem \ref{Theorem3} that McAdam's theorem  holds for $\QQ$-divisorial filtrations in an equidimensional local ring which is essentially of finite type over an excellent local ring of dimension less than or equal to 3.  This generalizes an earlier result for $\QQ$-divisorial filtrations in an equicharacteristic zero excellent local domain by the author. This theorem does not hold for more general filtrations.  

We consider the question of the asymptotic behavior of the function $n\mapsto \lambda_R(R/I_n)$ for a $\QQ$-divisorial filtration $\mathcal I=\{I_n\}$ of $m_R$-primary ideals on a $d$-dimensional normal excellent local ring. It is known from earlier work of the author that the multiplicity
$$
e(\mathcal I)=d!
\lim_{n\rightarrow\infty}\frac{\lambda_R(R/I_n)}{n^d}
$$
can be irrational. We show in Lemma \ref{Lemmadiff} that the limsup of the first difference function
$$
\limsup_{n\rightarrow\infty}\frac{\lambda_R(I_n/I_{n+1})}{n^{d-1}}
$$ 
is always finite for a $\QQ$-divisorial filtration. 
We then give an example in Section \ref{SecExample} showing that this limsup may not exist as a limit.

In the final section, we give an example of a symbolic filtration $\{P^{(n)}\}$ of a prime  ideal $P$ in a normal two dimensional excellent local ring which has the property that  the set of Rees valuations of all the symbolic powers $P^{(n)}$  of $P$ is infinite.

\end{abstract}


\maketitle

\section{Introduction}

\subsection{Graded Filtrations, Rees Algebras and Analytic Spread}

Let $R$ be a Noetherian local ring with maximal ideal $m_R$. $\lambda_R(M)$ will denote the length of an $R$-module $M$.

 If $I$ is an ideal in $R$, then the Rees algebra of $I$ is the graded $R$-algebra
$R[It]=\sum_{n\ge 0}I^nt^n$ and the analytic spread of $I$ is 
$$
\ell(I)=\dim R[It]/m_RR[It].
$$ 
The classical theory of the analytic spread of an ideal is expounded in \cite{HS}. A recent paper on related problems  using geometric methods is \cite{ORWY}. 

It is natural to extend the definition of analytic spread of an ideal to arbitrary graded filtrations of $R$. A graded filtration $\mathcal I=\{I_n\}$ of $R$ is a filtration of $R$ such that   $I_mI_n\subset I_{m+n}$ for all $m,n$.
Let $\mathcal I=\{I_n\}$ be a graded filtration on  $R$. The Rees algebra of  $\mathcal I$ is $R[\mathcal I]=\oplus_{n\ge 0}I_n$ and the analytic spread of the filtration $\mathcal  I$ is defined to be
\begin{equation}\label{eqN10}
\ell(\mathcal I)=\dim R[\mathcal I]/m_RR[\mathcal I].
\end{equation}
Thus if $I$ is an ideal in $R$ and $\mathcal I=\{I^n\}$ is the $I$-adic filtration, then $\ell(I)=\ell(\mathcal I)$.

We always have (\cite[Lemma 3.6]{CPS}) that 
\begin{equation}\label{eqZ1}
\ell(\mathcal I)\le \dim R,
\end{equation}
so the classical inequality $\ell(I)\le \dim R$ of ideals( \cite[Proposition 5.1.6]{HS}) extends to arbitrary graded filtrations.

In the case of the $I$-adic filtration $\mathcal I=\{I^n\}$ of an ideal in a  local ring $R$, the analytic spread $\ell(I)=e$ for some $e$ if and only if
$0<\lim_{n\rightarrow\infty}\frac{\lambda_R(I^n/m_RI^n)}{n^{e-1}}<\infty$.   This follows from the theory of Hilbert polynomials of a standard graded ring over a field. 

 This result fails spectacularly for general filtrations, as $\lambda_R(I_n/m_RI_n)$ can have arbitrarily large rates of growth on a fixed local ring $R$.  This can be seen in the following simple example, which was found with Hailong Dao and Jonathan Monta\~no. 

\begin{Example} Let $R=k[x,y,z]_{(x,y,z)}$  be the localization of a polynomial ring over a field $k$. Given any function $\sigma:\ZZ_{>0}\rightarrow\ZZ_{>0}$, define the graded filtration $\mathcal I=\{I_n\}$ by $I_0=R$ and $I_n=(z^{n+2})+z^{n+1}(x,y)^{\sigma(n)}$ for $n\ge 1$.
Then 
$$
\lambda_R(I_n/m_RI_n)=\sigma(n)+2.
$$
\end{Example}
The rate of growth of  $\lambda_R(I_n/m_RI_n)$ for    ``naturally occuring''  filtrations is known  to be polynomial in many cases. Shankar Dutta in \cite{Du} and  Hailong Dao and Jonathan Monta\~no in \cite{DM} have shown  that the growth of $\lambda_R(I^{(n)}/m_RI^{(n)})$ for the symbolic filtration $\{I^{(n)}\}$ of an ideal $I$ in a local ring $R$ has polynomial growth if $R$ and the modules $R/I^{(n)}$ have high depth for $n\gg 0$.

When $\mathcal I=\{I_n\}$ is an arbitrary graded family of $m_R$-primary ideals in a local ring $R$ of dimension $d$, then the rate of growth is always polynomial. We immediately find  the bound
$
\lambda_R(I_n/m_RI_n)\le cn^{d}
$
for some constant $c$ since  there is a power $b$ such that $m_R^b\subset I_1$, so that $m_R^{bn}\subset I_n$ for all $n$.

\subsection{McAdam's theorem for divisorial filtrations}
We have the following remarkable theorem by McAdam.

\begin{Theorem}\label{TheoremC4}(\cite{McA}, \cite[Theorem 5.4.6]{HS}) Let $R$ be a formally equidimensional local ring and $I$ be an ideal in $R$. Then $m_R\in \mbox{Ass}(R/\overline{I^n})$ for some $n$ if and only if $\ell(I)=\dim(R)$.
\end{Theorem}

The assumption of being formally equidimensional is not required for the if direction of Theorem \ref{TheoremC4} (this is Burch's theorem, \cite{Bu}, \cite[Proposition 5.4.7 ]{HS}).

To see how Theorem \ref{TheoremC4} can be extended, we  consider divisorial filtrations. Divisorial filtrations are defined in Section \ref{SecDiv}.
Divisorial filtrations arise naturally in ideal theory, as illustrated by the following classical theorem, following from work of Rees.

\begin{Theorem}\label{ReesThm} Suppose that $I$ is an ideal in an equidimensional excellent local ring $R$. Then the filtration $\{\overline{I^n}\}$ of $R$ is a divisorial filtration.
\end{Theorem}

\begin{proof} This follows from \cite{Re} or the proof of  \cite[Theorem 10.2.2]{HS},  \cite[Proposition 10.4.3]{HS} and  \cite[Scholie IV.7.8.3 (x)]{EGAIV}.
\end{proof}

The Rees valuations of an ideal $I$ are the divisorial valuations occuring in an irredundant representation of the filtration $\{\overline{I^n}\}$ as a divisorial filtration.

Another important example of a divisorial filtration is the filtration of symbolic powers $\{P^{(n)}\}$ of a prime ideal in a local ring $R$ such that $R_P$ is a regular local ring. The paper \cite{DDGHN} gives a survey of some recent results on symbolic algebras.

The following theorem is proven for local domains in \cite{CPS}. The proof is the same for local rings, using the notation on valuations of a ring and divisorial filtrations given in Section \ref{SecDiv}.

\begin{Theorem}\label{PropAS7*}(Theorem 1.4 \cite{CPS} extended to local rings) Suppose that $R$ is a  local ring and $\mathcal I=\{I_n\}$ is an $\RR$-divisorial filtration on $R$. 
%Let $P_1,\ldots,P_r$ be the minimal prime ideals of $R$ and  let $\phi_i:R\rightarrow R/P_i$ be the natural projection  for $1\le i\le r$.Then there exist divisorial valuations   $\{v_{ij}\}$ be divisorial valuations on $R/P_i$ and  $\lambda_{ij}\in \RR_{\ge 0}$ such that 
%$$
%I_n=\cap_{i=1}^r\cap_{i,j} \phi_i^{-1}(I(v_{ij})_{n\lambda_{ij}})
%$$
%for all $n$.
  %Then for some $v_{ij}$, the center $\phi^{-1}m_{v_{ij}}=\{f\in R\mid \nu_i(\phi_i(f))>0\}$   is $m_R$.
 	%Then there exists a positive integer $n_0$ such that for $n\ge n_0$,
	Then
	\begin{enumerate}
	\item[a)]$I_n=\overline{I_n}$ for all $n>0$ and
	\item[b)] Suppose that $\ell(\mathcal I)=\dim R$.  
Then there exists a positive integer $n_0$ such that $m_R$ is an associated prime of $I_n$ for all $n\ge n_0$.
	\end{enumerate}
		\end{Theorem}
	
	The following theorem is a generalization of Theorem \ref{TheoremC4} to divisorial filtrations. If $\mathcal I=\{I_n\}$ is a divisorial filtration on an excellent local ring $R$, then $I_n=\overline{I_n}$ for all $n$ and $R$ is formally equidimensional if $R$ is equidimensional (Scholie IV.7.8.3 (x) \cite{EGAIV}).
	
\begin{Theorem}\label{Theorem3} Let $R$ be an equidimensional local ring which is essentially of finite type over an excellent local ring of dimension less than or equal to 3. Let $\mathcal I=\{I_n\}$ be a $\QQ$-divisorial filtration on $R$. Then the following are equivalent.
\begin{enumerate}
\item[1)] The analytic spread of $\mathcal I$ is $\ell(\mathcal I)=\dim R$.
\item[2)] There exists $n_0\in \ZZ_{>0}$ such that $m_R\in \mbox{Ass}(R/I_n)$ if $n\ge n_0$.
\item[3)] $m_R\in \mbox{Ass}(R/I_{m_0})$ for some $m_0\in \ZZ_{>0}$.
\end{enumerate}
\end{Theorem}

3) implies 1) of Theorem \ref{Theorem3} is Theorem \ref{Theorem2} and Temkin's theorem on existence of regular alterations \cite[Theorem 1.2.5]{Te2}. Temkin's theorem shows that if $X$ is an integral scheme which is of finite type over an excellent domain of dimension less than or equal to three, then there exists a projective regular alteration $X'\rightarrow X$. Alterations and regular alterations are defined in Definition \ref{AltDef} and just after this definition.
1) implies 2) follows from Theorem \ref{PropAS7*}.

Theorem \ref{Theorem3} is proven for excellent equicharacteristic zero local domains in \cite[Theorem 1.4]{SDCProc}, using the existence of resolutions of singularities of excellent equicharacteristic zero domains (\cite{Hi} or \cite{Te}). The proof in \cite{SDCProc} shows that Theorem \ref{Theorem3} is true  for excellent local domains of dimension $\le 3$, using the existence of resolutions of singularities of  excellent domains of dimension $\le 3$ proven in \cite{CP}.

The conclusions of 3) implies 1) of Theorem \ref{Theorem3} do not hold for more general filtrations. In fact, there exist $\RR$-divisorial filtrations $\mathcal I$ on an excellent equicharacteristic zero local domain such that 
the conclusions of Theorem \ref{Theorem2} are false, as shown in Example 1.5 \cite{SDCProc}.

The existence of regular alterations over a local domain which is essentially of finite type over a field was first proven by de Jong
in \cite[Theorem 4.1]{dJ}.

\begin{Remark} The conclusions of Theorem \ref{Theorem3} hold whenever $R$ is an equidimensional excellent local ring which has the property that if $P$ is a minimal prime of $R$ and $X\rightarrow \mbox{Spec}(R/P)$ is a birational projective morphism with $X$ integral, then there exists a regular alteration $X'\rightarrow X$; in fact, the above proof of Theorem \ref{Theorem3} (replacing  the reference to Temkin's theorem with this assumption) proves this  result.  
\end{Remark}

%\begin{Lemma}\label{valuationlemma2}(Lemma 4.3 \cite{SDCProc} extended to local rings) Suppose that $R$ is a local ring and $\mathcal I=\{I_n\}$ is a divisorial filtration on $R$. Suppose that $P$ is a prime ideal of $R$ and there exists $t\in \ZZ_{>0}$ such that $P\in \mbox{Ass}(R/I_t)$. Then there exists $n_0\in \ZZ_{>0}$ such that $P\in \mbox{Ass}(R/I_n)$ for all $n\ge n_0$.
% \end{Lemma}
 
 %\begin{Theorem}\label{PropAS7}(Theorem 4.7 \cite{SDCProc} extended to  local rings) Suppose that $R$ is a  local ring and $\mathcal I=\{I_n\}$ is a divisorial filtration on $R$. Let $I_n=I(\nu_1)_{a_1n}\cap\cdots\cap I(\nu_r)_{a_rn}$ for $n\ge 1$, some valuations $\nu_i$ and some $a_1,\ldots,a_r\in \ZZ_{>0}$. Suppose that $\ell(\mathcal I)=\dim R$. Then for some $\nu_i$, the center $m_{\nu_i}\cap R=\{f\in R\mid \nu_i(f)>0\}$   is $m_R$. There exists a positive integer $n_0$ such that $m_R$ is an associated prime of $I_n=\overline{I_n}$ for all $n\geq n_0$.
	%\end{Theorem}		
		
\subsection{Hilbert functions of divisorial $m_R$-filtrations}
If $I$ is an ideal in a Noetherian local ring $R$, the function $\lambda_R(I^n/m_RI^n)$ is a polynomial for large $n$, whose degree is the analytic spread $\ell(I)$ of $I$. The computation of asymptotic properties of the function $\lambda_R(I_n/m_RI_n)$ is difficult for a divisorial filtration $\mathcal I=\{I_n\}$. In the case that $R$ is normal and excellent of dimension two, we show in \cite[Theorem 1.6]{C1} that for a divisorial filtration $\mathcal I$ of $R$, there exists $\alpha\in \QQ_{\ge 0}$ such that 
$$
\lambda_R(I_n/m_RI_n)=n\alpha+\sigma(n)
$$
for $n>0$, where $\sigma(n)$ is a bounded function of $n$. We have that $\alpha>0$ if and only if $\ell(\mathcal I)=2=\dim R$.
The proof uses Zariski decomposition, which does not exist in general if $\dim R\ge 3$. 

Now suppose that $\mathcal I$ is a divisorial filtration of $m_R$ ideals. Then, related to the function $\lambda_R(I_n/m_RI_n)$, we have the functions 
\begin{equation}\label{Hilbfun}
n\mapsto \lambda_R(R/I_n).
\end{equation}
and its first difference $n\mapsto \lambda_R(I_n/I_{n+1})$.
A good property of the function (\ref{Hilbfun}) is that the multiplicity 
$$
e(\mathcal I)=\lim_{n\rightarrow\infty}
d!\frac{\lambda_R(R/I_n)}{n^d}
$$
always exists.  In fact, this  limit always exists for arbitrary graded families of $m_R$-primary ideals in analytically unramified local rings (\cite[Theorem 1.1]{C2} and  \cite[Theorem 4.2]{CAs}). This theorem builds on earlier work by   \cite{ELS}, \cite{Mus}, \cite{Ok}, \cite{KK} and \cite{LM}.

The multiplicity $e(\mathcal I)$ is related to the analytic spread $\ell(\mathcal I)$ as follows. 

\begin{Lemma} Suppose that $R$ is a $d$-dimensional excellent normal local ring and $\mathcal I=\{I_n\}$ is a divisorial filtration of $m_R$-primary ideals of $R$. Then $e(\mathcal I)>0$ and $\ell(\mathcal I)=d$.
\end{Lemma}

We give an outline of the proof. $\ell(\mathcal I)=d$ by Theorem \ref{Theorem3}. There exist divisorial valuations $w_j$ on $R$ (the Rees valuations of $m_R$) such that $\mbox{center}_R(w_j)=m_R$ for all $j$, and positive integers $b_j$ such that 
$$
\overline{m_R^n}=\cap_jI(w_j)_{mb_j}\mbox{ for all $n\ge 0$}.
$$
Writing 
$$
I_n=\cap_jI(v_i)_{na_i}
$$
for suitable divisorial valuations $v_i$ with $\mbox{center}_R(v_i)=m_R$ and $a_i\in \QQ$, we see by Izumi's Theorem (\cite{Re2}) that there exists a positive integer $\beta$ such that $I_{\beta n}\subset \overline{m_R^n}$ for all $n>0$. Now 
the ring $\sum_{n\ge 0}\overline{m_R^n}t^n$ is a finite $R[m_Rt]$-module by \cite[Proposition 5.2.1]{HS} and \cite[Corollary 9.2.1]{HS}, so
$$
0<\lim_{n\rightarrow\infty}\frac{\lambda_R(R/\overline{m_R^n})}{n^d}
\le \lim_{n\rightarrow\infty}\frac{\lambda_R(R/I_n)}{n^d}=e(\mathcal I).
$$

In this subsection, 
we  consider the asymptotic properties of the difference function $n\mapsto \ell_R(I_n/I_{n+1})$. In  the  case of an $I$-adic filtration, this function is a polynomial of degree $d-1$ for $n\gg0$.
We always have that
\begin{equation}\label{diff1}
\limsup_{n\rightarrow\infty}\frac{\lambda_R(I_n/I_{n+1})}{n^{d-1}}
\end{equation}
is finite for $\QQ$-divisorial filtrations of $m_R$-primary ideals, as is shown in Lemma \ref{Lemmadiff}.  However, there are examples of graded families of ideals such that this limsup is infinite (\cite[Theorem 4.3]{CAs}).  Even when this limsup is finite, it may be that the limsup is not a limit for graded filtrations of $m_R$-primary ideals (\cite[Theorem 4.6]{CAs}). 

The function (\ref{Hilbfun}) is  very nice in dimension 2. 
In Section 8 \cite{CSr}, it is shown that if  $R$ is a two dimensional normal excellent local ring and $\mathcal I=\{I_n\}$ is a $\QQ$-divisorial filtration of $m_R$-primary ideals on $R$, then there exists a quadratic polynomial $P(n)$ with rational coefficients, a function $f(n)$  and a constant $c$ such that
$$
\lambda_R(R/I_n)=P(n)+f(n)
$$
with $|f(n)|<c$ for all $n$.
In  \cite[Theorem 9]{CSr}, it is shown that if $R$ has equicharacteristic zero then $f(n)$ is eventually periodic. Example 5 \cite{CSr} shows that $f(n)$ may not be eventually periodic if $R$ has characteristic $p>0$.

It follows that in  dimension two, the multiplicity
$e(\mathcal I)$ is always a rational number.  Further, the limsup  (\ref{diff1}) is a limit, and this limit is a rational number.

In Example 6 \cite{CSr} and in \cite{C}, 
examples are given of a three dimensional normal local domain $R$ with an isolated singularity and a divisorial valuation $v$ on $R$ such that the center of $v$ on $R$ is the maximal ideal $m_R$ and 
the multiplicity $e(\mathcal I)$ of the filtration $\mathcal I=\{I_n\}$ where $I_n=\{x\in R\mid v(x)\ge n\}$ is irrational.

By (\ref{diff1}), the limsup of difference functions always exists. 
We consider in Section \ref{SecExample} the question of if this limsup is always a limit; that is, does the limit
\begin{equation}\label{diff2}
 \lim_{n\rightarrow\infty}\frac{\lambda_R(I_n/I_{n+1})}{n^{d-1}}
\end{equation}
of a $\QQ$-divisorial filtration of $m_R$-primary  ideals always exist?

As we have seen above, this limit does always exist and is even a rational number when the dimension of $R$ is $2$.

We give in Section \ref{SecExample} an example showing  that the limit (\ref{diff2}) does not exist for a $\QQ$-divisorial filtration on a three dimensional normal
local ring $R$ which is essentially of finite type over an algebraically closed field of characteristic zero.

To accomplish this, we start with the construction of an example with irrational multiplicity in \cite{C}. In this example, a resolution of singularities $U\rightarrow \mbox{Spec}(R)$ of $\mbox{Spec}(R)$ is constructed which 
has two exceptional divisors, $\overline S$ and $F$ both of which contract to $m_R$. Let $v$ be the canonical valuation associated to the valuation ring $\mathcal O_F$ and $\mathcal I=\{I_n\}$ where $I_n=I(v)_n=\{x\in R\mid v(x)\ge n\}$, which is a $\ZZ$-divisorial filtration on $R$.

It is shown in the formula before  \cite[Lemma 2.2]{C} and   \cite[Theorem 4.1]{C} that
$$
I_n=I(nF):=\Gamma(U,\mathcal O_U(-nF))=\Gamma(U,\mathcal O_U(-D_n))
$$ 
where  $D_n=\lceil n\alpha\rceil \overline S+nF$, $\alpha=\frac{3}{9-\sqrt{3}}$.    It is further shown that the divisors $-D_n$ are nef divisors on $U$.
Thus by formula (9) and   \cite[Proposition 2.4]{C}, 
$$
\lim_{n\rightarrow\infty}\frac{\lambda_R(R/I(nF))}{n^3}=-\frac{((-\alpha \overline S-F)^3)}{3!}=\alpha^3468-486\alpha^2+162\alpha+54,
$$
so that 
\begin{equation}\label{form30}
e(\mathcal I)=3!(\alpha^3468-486\alpha^2+162\alpha+54)=\frac{72252}{169}-\frac{162}{169}\sqrt{3}
\end{equation}
is an irrational number. Equation (\ref{form30}) corrects equation (12) of \cite{C}, where $\alpha^3468-486\alpha^2+162\alpha+54$ is incorrectly calculated to be $\frac{12042}{169}-\frac{27}{169}\sqrt{3}$.

To compute the second difference $\lambda_R(I(nF)/I(n+1)F)$, we must make a finer calculation, computing the function $n\mapsto n\lambda_R(R/I_n)$ explicitely up to a constant times $n$.

In (\ref{form10}) it is shown that		  
\begin{equation}\label{form10*}
\lambda_R(R/I(nF))=\frac{1}{6}(D_n^3)+\frac{1}{4}(D_n^2\cdot K_Y)+\delta(n)
\end{equation}
for $n\in \NN$ where there is a constant $c_1$ such that $|\delta(n)|<c_1n$ for all $n\in \NN$. By (\ref{form5}) and (\ref{form22}), this length has the numeric formulation
$$
\begin{array}{lll}
\lambda_R(R/I(nF))&=&\frac{486}{6}\lceil \alpha n\rceil^3-\frac{486}{6}n\lceil \alpha n\rceil^2+\frac{162}{6}n^2\lceil \alpha n\rceil+\frac{54}{6}n^3\\
&&-\frac{792}{4}\lceil n\alpha\rceil^2+\frac{564}{4}n\lceil n\alpha\rceil-\frac{174}{4}n^2+\delta(n).
\end{array}
$$
So taking the limit in this equation, we obtain 
\begin{equation}\label{form7*}
e(\mathcal I(F))=3!\left(\lim_{n\rightarrow\infty}\frac{\lambda_R(R/I(nF))}{n^3}\right)=3!(\alpha^3468-486\alpha^2+162\alpha+54)
=\frac{72252}{169}-\frac{162}{169}\sqrt{3}
\end{equation}
obtaining (\ref{form30}) by a different method.

%(corrected from $\frac{313092}{169}-\frac{702}{169}\sqrt{3}$)
%giving the (corrected) result of Theorem 1.4 \cite{C}, which was calculated by a different method.

From (\ref{form10*}), we calculate in (\ref{form8}) that 
\begin{equation}\label{form8*}
\lambda_R(I(nF)/I((n+1)F))=\lambda_R(R/I((n+1)F))-\lambda_R(R/I(nF))=\frac{1}{6}\left((D_{n+1}^3)-(D_n^3)\right)+\epsilon(n)
\end{equation}
where there is a constant $c''$ such that $|\epsilon(n)|<c''n$ for $n\in\NN$. At the end of Subsection \ref{SecExample1}, it is shown that $\ZZ_{>0}$ is the union $\ZZ_{>0}=\Sigma_1\coprod \Sigma_2$ where both $\Sigma_1$ and $\Sigma_2$ are infinite, 
the limit as $n\rightarrow\infty$ restricted to $n\in \Sigma_1$ of $\frac{\lambda_R(I(nF)/I((n+1)F)}{n^2}$ is 
$$
\frac{1}{6}\left(-486\alpha^2+324\alpha+162\right)
=\frac{144504}{6(26)^2}-\frac{324}{6(26)^2}\sqrt{3}
$$
and the limit as $n\rightarrow\infty$ restricted to $n\in \Sigma_2$ of $\frac{\lambda_R(I(nF)/I((n+1)F)}{n^2}$ is
$$
\frac{1}{6}\left(918\alpha^2-810\alpha+324\right)
=\frac{106596}{6(26)^2}-\frac{4536}{6(26)^2}\sqrt{3}.
$$
These two limits, over $n$ in $\Sigma_1$ and $\Sigma_2$ respectively are not equal, so the limit
$$
\lim_{n\rightarrow\infty}\frac{\lambda_R(I_n/I_{n+1})}{n^2}
$$
does not exist.

\subsection{Symbolic powers with infinitely many Rees valuations} 
Let $R$ be a $d$-dimensional excellent normal local ring. Let $\mathcal I=\{I_n\}$ be a graded family of ideals of $R$ ($\mathcal I$ is a filtration of $R$ and $I_mI_n\subset I_{m+n}$ for all $m,n$). We have that $\sqrt{I_n}=\sqrt{I_1}$ for all $n$. 
We define the Rees valuations of $\mathcal I$ to be the divisorial valuations $v$ of $R$ such that the center of $v$ on $\mbox{Proj}(R[\mathcal I])$ has dimension $d-1$ and the center $\mbox{center}_R(v)$ of $v$ on $R$ contains $I_1$ (and so contains $I_n$ for all $n$).
In the case that $\mathcal I=\{I^n\}$ is the $I$-adic filtration for some ideal $I$ of $R$, this agrees with the classical definition of the Rees valuations of $I$ (\cite[Chapter 10]{HS}).

Suppose that $\mathcal I$ is a divisorial filtration with $I_n=I(v_1)_{n\lambda_1}\cap \cdots\cap I(v_r)_{n\lambda_r}$ for all $n$.
Then the Rees valuations of $\mathcal I$ are a subset of $\{v_1,\ldots,v_r\}$. The Rees valuations can be a proper subset of 
$\{v_1,\ldots,v_r\}$. Theorem \ref{Theorem3} shows that if $R$ is essentially of finite type over a field (in addition to our assumption that $R$ is normal and excellent) then the analytic spread $\ell(\mathcal I)=d$ if and only if there exists a Rees valuation $v$ of $\mathcal I$ such that $\mbox{center}_R(v)=m_R$.

We have the natural birational maps 
$$
\mbox{Proj}(R[I_nt^n])\dashrightarrow \mbox{Proj}(R[\mathcal I])
$$
for all $n>0$. The nature of these birational maps (and hence the structure of $\mbox{Proj}(R[\mathcal I])$) is determined by the respective Rees valuations of $\mbox{Proj}(R[I_nt^n])$ and $\mbox{Proj}(R[\mathcal I])$.

This raises the question of what asymptotic properties the Rees valuations of $I_n$ have (as $n$ goes to infinity) and how this compares with the Rees valuations of $\mathcal I$.

If $I$ is an ideal in an excellent local ring, then we have seen in Theorem \ref{ReesThm} that the filtration of powers of integral closures $\{\overline{I^n}\}$ of the powers of $I$ is a $\ZZ$-divisorial filtration. In particular, the set of all Rees valuations of all of the ideals $I^m$ over all $m$ is a finite set. It is just the set of Rees valuations of $I$ itself. 
In Section \ref{SecSym} we consider the question of if this set is finite for the ideals in a $\ZZ$-divisorial filtration; that is, is the set of all Rees valuations of the ideals $I_n$ finite when $\mathcal I=\{I_n\}$ is a $\ZZ$-divisorial filtration? We show that this is not the case.
In Section \ref{SecSym} we construct an example of a height one prime ideal $P$ in a two dimensional normal local ring $R$ such that for all $n\ge 1$, the symbolic power $P^{(n)}$ has a Rees valuation other then the $PR_P$-adic  valuation $v_P$, but for  $1\le m<n$, $v_P$ is the only common Rees valuation of $P^{(m)}$ 
and $P^{(n)}$. In particular, the set of Rees valuations of $P^{(n)}$ over all  $n\ge 1$ is infinite. 

The only Rees valuation of the divisorial filtration $\{P^{(n)}\}$ is $v_P$ and the analytic spread of this filtration is
$\ell(\{P^{(n)}\})=0$. In contrast, the analytic spread of each individual ideal $P^{(n)}$ is maximal, with $\ell(P^{(n)})=2$. 
None of the Rees valuations of the $P^{(n)}$ which have center $m_R$ on $R$ have a center on $\mbox{Proj}(R[\mathcal I])=\mbox{Spec}(R)\setminus\{m_R\}$. Necessarily, $R[\mathcal I]$ is not Noetherian.

\subsection{Acknowledgements} The author would like to thank Hailong Dao and Jonathan Monta\~no  for discussions and questions about Hilbert polynomials of filtrations and to thank Linquin Ma,  Thomas Polstra and  Karl Schwede     for a discussion about the cardinality (finite or infinite) of the set of Rees valuations of the symbolic powers of a prime ideal. He thanks Suprajo Das for suggesting the use of alterations to generalize the assumption that rings have equicharacteristic zero in 
\cite[Theorem 1.3]{SDCProc}. He thanks the reviewer for helpful comments to improve the exposition. 

\section{Valuations, Filtrations and Rees Algebras}

A local ring will be assumed to be Noetherian.
We will denote the maximal ideal of a local ring $R$ by $m_R$. The length of an $R$-module $M$ will be denoted by $\lambda_R(M)$ or $\lambda(M)$.

\subsection{Extensions of valuations and Rees algebras}

Let $K$ be a field  and $v$ be a rank 1 valuation of $K$ (the value group $vK$ of $v$ is order isomorphic to a subgroup of $\RR$). For $\lambda\ge 0$ in the value group $vK$, let 
$$
I(v)_{\lambda} = \{x\in K\mid v(c)\ge \lambda\},
$$
an ideal in the valuation ring $\mathcal O_v$ of $v$.
Let $R$ be a domain with quotient field $K$ such that $v$ is nonnegative on $R$. The center of $v$ on $R$ is the prime ideal 
$$
\mbox{center}_R(v)=m_v\cap R,
$$
 where $m_v$ is the maximal ideal of $\mathcal O_v$. Let $P=\mbox{center}_R(v)$. 
 We have that  $\sqrt{I(v)_{\lambda}\cap R}
=P$ for all $\lambda\in vK_{>0}$ and in fact $I(v)_{\lambda}\cap R$ is $P$-primary, since the inclusion $R\rightarrow \mathcal O_v$ factors through $R_P$ and $m_v\cap R_P=PR_P$.

The following lemma is immediate. 

\begin{Lemma}\label{LemmaA}
Let $K\rightarrow L$ be a finite extension of fields, $v_1,\ldots,v_r$ be rank 1 valuations of $K$, $\lambda_i\in v_iK$  and $\{w_{ij}\}$ be the extensions of $v_i$ to $L$ for $1\le i\le r$. 
Then 
$$
(\cap_i\cap_j I(w_{ij})_{\lambda_i})\cap K=\cap_iI(v_i)_{\lambda_i}.
$$
\end{Lemma}

%\begin{proof} Suppose that $x\in K$ and $v_i(x)\ge\lambda_i$ for all $i$. Then $w_{ij}(x)\ge \lambda_i$ for all $i,j$ so that 
%$\cap_iI(v_i)_{\lambda_i}\subset (\cap_i\cap_j I(w_{ij})_{\lambda_i})\cap K$. Now suppose that $x\in K$ and $w_{ij}(x)\ge \lambda_i$ for all $i,j$.
%Then $v_i(x)\ge \lambda_i$ for al $i$, so that $x\in \cap_iI(v_i)_{\lambda_i}$.
%\end{proof}

\begin{Lemma}\label{LemmaB} Let $K\rightarrow L$ be a finite field extension, $R$ be a normal Noetherian domain with quotient field $K$ and $S$ be the integral closure of $R$ in $L$. %Suppose that $S$ is a finite $R$-module. 
Let $v$ be a valuation of $K$ which is nonnegative on $R$, and $w$ be an extension of $v$ to $L$. Then $w$ is nonnegative on $S$ and 
$$
\dim S/\mbox{center}_S(w)=\dim R/\mbox{center}_R(v).
$$
In particular, $\mbox{center}_R(v)$ is a maximal ideal of $R$ if and only of $\mbox{center}_S(w)$ is a maximal ideal of $S$.
\end{Lemma}

\begin{proof}  The maximal ideal $m_w$ of $\mathcal O_w$ contracts to the maximal ideal $m_v$ of $\mathcal O_v$ so $\mbox{center}_S(w)$ contracts to $\mbox{center}_R(v)$. Thus the lemma holds since $S$ is integral over $R$.
\end{proof}

For ideals $U$ and $V$ of a  ring $R$, define 
$$
U:V^{\infty}=\cup_{n>0}U:V^n=\{f\in R\mid fV^n\subset U\mbox{ for some $n>0$}\}.
$$
If $R$ is Noetherian, 
$U:V^{\infty}$ is obtained from $U$ by removing from a primary decomposition of $U$ the $P$-primary components that have the property that $J(R)\subset P$.

Let $R$ be a normal Noetherian semi local domain with quotient field $K$, $v_1,\ldots,v_r$ be rank 1 valuations of $K$ which are nonnegative on $R$ and $\lambda_i\in v_iK_{\ge 0}$ for $1\le i\le r$.
Let $J(R)$ be the intersection of the maximal ideals of $R$.  

We have that 
$$
\left(\cap_{i=1}^r I(v_i)_{\lambda_i}\right)\cap R=\cap_{i=1}^r(I(v_i)_{\lambda_i}\cap R).
$$
Let $P_j$ for $1\le j\le s$ be the distinct centers of the $v_i$ on $R$, and define $\sigma(i)$ by $\mbox{center}_R(v_i)=P_{\sigma(i)}$ for $1\le i\le r$. As observed before Lemma \ref{LemmaA}, $I(v_i)_{\lambda_i}\cap R$ is $P_{\sigma(i)}$-primary for all $i$. Let $Q_j=\cap_{\sigma(i)=j}(I(v_i)_{\lambda_i}\cap R)$, which is a $P_j$-primary ideal since it is an intersection of $P_j$-primary ideals. Thus $\cap_{j=1}^s Q_j$ is a primary decomposition of $\cap_{i=1}^r(I(v_i)_{\lambda_i}\cap R)$. As remarked above, $(\cap_{j=1}^sQ_j):J(R)^{\infty}$ removes from $\cap_{j=1}^s$ the primary components from the maximal ideals of $R$.

Thus 
\begin{equation}\label{eqN4}
\left[\cap_{i\in \Lambda}I(v_i)_{\lambda_i}\right]\cap R
=\left(\left[\cap_{i=1}^rI(v_i)_{\lambda_i}\right]\right)\cap R:J(R)^{\infty},
\end{equation}
where 
$$
\Lambda=\{i\mid \mbox{center}_R(v_i)\mbox{ is not a maximal ideal of $R$}\}.
$$

The following lemma follows from Lemma \ref{LemmaA}, Lemma \ref{LemmaB} and equation (\ref{eqN4}).

\begin{Lemma}\label{LemmaC} 
Let $K\rightarrow L$ be a finite field extension, $R$ be a normal Noetherian domain with quotient field $K$ and $S$ be the integral closure of $R$ in $L$.  Suppose that $S$ is a finite $R$-module. Let  $v_1,\ldots,v_r$ be rank 1 valuations of $K$, $\lambda_i\in v_iK$ be nonnegative  and $\{w_{ij}\}$ be the extensions of $v_i$ to $L$ for $1\le i\le r$. Suppose that $v_1,\ldots,v_r$ are nonnegative on $R$.
Then
$$
\left(\left[\cap I(w_{ij})_{\lambda_i}\right]\cap S:J(S)^{\infty}\right)\cap R=
\left[\cap I(v_i)_{\lambda_i}\right]\cap R:J(R)^{\infty}.
$$
 In particular, 
 $$
 \left[\cap I(v_i)_{\lambda_i}\right]\cap R:J(R)^{\infty}\ne  \left[\cap I(v_i)_{\lambda_i}\right]\cap R
 $$
  implies
 $$
 \left[\cap I(w_{ij})_{\lambda_i}\right]\cap S:J(S)^{\infty}\ne  \left[\cap I(w_{ij})_{\lambda_i}\right]\cap S.
 $$
 %Suppose that  $[\cap I(w_{ij})_{\lambda_in}]\cap S:J(S)^{\infty}\ne [\cap I(w_{ij})_{\lambda_in}]\cap S$ for $n\ge n_0$. Then there exists $l_0$ such that 
 %$[\cap I(v_i)_{\lambda i n}]\cap R:J(R)^{\infty}\ne [\cap I(v_i)_{\lambda_in}]\cap R$ for $n\ge l_0$.
 \end{Lemma}

  Let $K$ be a field  and $v_1,\ldots,v_r$ be rank 1 valuations of $K$ and $\lambda_i\ge 0$ be elements of $v_iK$ for $1\le i\le r$.
  Let $R$ be a normal semilocal domain with quotient field $K$ such that the $v_i$ are nonnegative on $R$ for all $i$. Define a graded filtration $\mathcal I(R)=\{I(R)_n\}$ of $R$ by $I(R)_n=(\cap_iI(v_i)_{n\lambda_i})\cap R$. Suppose that $K\rightarrow L$ is a finite field extension. Let $\{w_{ij}\}$ be the extensions of $v_i$ to $L$ for $1\le i\le r$. Let $S$ be the integral closure of $R$ in $L$. 
 Define a graded filtration
   \begin{equation}\label{eq5}\mbox{ $\mathcal I(S)=\{I(S)_n\}$ on $S$ by 
  $I(S)_n=(\cap_i\cap_jI(w_{ij})_{n\lambda_i})\cap S$.}
  \end{equation}
  
 %The Rees algebra of a graded filtration $\mathcal J=\{J_n\}$ of a ring $T$ is 
 %$$
 %T[\mathcal J]=\oplus_{n\ge 0}J_n.
 %$$
 
 %Suppose that $\mathcal I=\{I_n\}$ is a filtration of a local ring $R$. Then the associated graded rings
  %$$
  %R[\mathcal I]=\sum_{n\ge 0} I_nt^n\mbox{ and }
    %S[\mathcal I]=R[\mathcal I][t^{-1}]
  %$$
   %are subrings of the graded ring $R[t,t^{-1}]$. We have 
  %a graded ring 
  %$$
  %T_{\mathcal I}:=R[\mathcal I]/m_RR[\mathcal I].
  %$$

  \begin{Proposition}\label{PropD}  Let $L$ be a finite field extension of $K$ and $S$ be the integral closure of $R$ in $L$.     Then $S[\mathcal I(S)]$ is integral over $R[\mathcal I(R)]$.
  \end{Proposition}
  
  \begin{proof}
 Let $N$ be a finite extension of $L$ such that $N$ is a normal extension of $K$. Let $G=\mbox{Aut}_K(N)$. We have a tower of fields
 $K\rightarrow N^G\rightarrow N$ where $N^G$ is the fixed field of $N$ under the action of $G$. The extension $K\rightarrow N^G$ is purely inseparable and $N^G\rightarrow N$ is Galois with Galois group $G$ (c.f.  \cite[Proposition V.6.11]{La}).
 
 Let $A$ be the integral closure of $R$ in $N^G$ and $B$ be the integral closure of $R$ in $N$. To prove the proposition, it suffices to show that $B[\mathcal I(B)]$ is integral over $R[\mathcal I(R)]$, and to prove this it suffices to prove that $B[\mathcal I(B)]$ is integral over $A[\mathcal I(A)]$ and $A[\mathcal I(A)]$ is integral over $R[\mathcal I(R)]$. Thus we are reduced to proving the Proposition in two cases: when $L$ is Galois over $K$ and when $L$ is purely inseparable over $K$.
 
 First suppose that $L$ is Galois over $K$, with Galois group $G$. If $x\in S^G$, then $x\in K$, and since $x$ is integral over $R$, $x\in R$. Thus $R=S^G$. Let 
 $$
 G=\{\sigma_1={\rm id}, \sigma_2,\ldots,\sigma_a\}.
 $$
 
 Let $x\in S$. Let  $f(X)=\prod_{i=1}^a(X-\sigma_i(x))\in K[X]$. We have that $f(x)=0$. Expand 
 \begin{equation}\label{eq1}
 f(X)=X^a-s_1X^{a-1}+s_2X^{a-2}+\cdots +(-1)^as_a
 \end{equation}
  where $s_i$ are the elementary symmetric polynomials in $\sigma_1(x),\ldots,\sigma_a(x)$. The $s_i$ are integral over $R$, and are contained in $K$, so that the $s_i$ are in $R$ since $R$ is normal.    
    
 Suppose  that $x\in I(S)_n\subset S$. Then 
 \begin{equation}\label{eq2}
 w_{i,j}(x)\ge n\lambda_i
 \end{equation}
  for all $i,j$. There exists a permutation $\tau_i$ of $\{1,\ldots,r\}$ such that 
  \begin{equation}\label{eq3}
  w_{i,1}\sigma_j=w_{i,\tau_i(j)}
    \end{equation}
   for all $j$ (\cite[Corollary VI.7.3]{ZS2}). For $b\in \{1,\ldots, a\}$, we compute 
 $$
 v_i(s_b)=w_{i,1}(s_b)\ge \min\{w_{i,1}(\sigma_{\alpha_1}(x))+\cdots+w_{i,1}(\sigma_{\alpha_b}(x))\}
 $$
 where the minimum is over $\alpha_1<\alpha_2<\cdots<\alpha_b$. Thus
\begin{equation}\label{eq4}
 v_i(s_b)\ge bn\lambda_i
 \end{equation}
 by (\ref{eq2}) and  (\ref{eq3}), so that $s_b\in I(R)_{bn}$. 
 
 Write $R[\mathcal I(R)]\cong \sum_{i=0}^{\infty}I(R)_nt^n$, a graded subring of the graded polynomial ring $R[t]$. Let $g(X)=\prod_{i=1}^a(X-\sigma_i(x)t^n)$. 
 Expand
 $$
 g(X)=X^a-s_1t^{n}X^{a-1}+s_2t^{2n}X^{ra2}+\cdots +(-1)^as_a t^{an}.
 $$
 We have that $(-1)^bs_bt^{bn}\in I(R)_{bn}t^{bn}$ for all $b$ by (\ref{eq4}). Thus $g(X)\in R[\mathcal I(R)][X]$.  Since
 $$
 g(xt^n)=\prod_{i=1}^a(xt^n-\sigma_i(x)t^n)=t^{an}\prod_{i=1}^a(x-\sigma_i(x))=0,
 $$
 this is an equation of integral dependence of $xt^n$ over $R[\mathcal I(R)]$. Thus $S[\mathcal I(S)]$ is integral over $R[\mathcal I(R)]$ in the case that $L$ is Galois over $K$.
 
 Now suppose that $L$ is purely inseparable over $K$. Let $[L:K]=p^m$ where $p$ is the characteristic of $K$. If $x\in L$, we have that $x^{p^m}\in K$.
 There is a unique extension $w_i$ of $v_i$ to $L$. This valuation satisfies $w_i(x)=\frac{1}{p^m}v_i(x^{p^m})$ for $x\in L$. Thus $x\in I(S)_n$ implies $x^{p^m}\in I(R)_{np^m}$ and so $(xt^n)^{p^m}\in I(R)_{np^m}t^{np^m}$. Let $g(X)=X^{p^m}-(xt^n)^{p^m}\in R[\mathcal I(R)][X]$. We have that $g(xt^n)=0$, giving an equation of integral dependence of $xt^n$ over $R[\mathcal I(R)]$. Thus $S[\mathcal I(S)]$ is integral over $R[\mathcal I(R)]$ in the case that $L$ is purely inseparable over $K$.
 
\end{proof} 

\subsection{Divisorial Filtrations}\label{SecDiv}

 A divisorial valuation of a Noetherian domain $R$ (\cite[Definition 9.3.1]{HS}) is a valuation $\nu$ of $K$ such that if $\mathcal O_{\nu}$ is the valuation ring of $\nu$ with maximal ideal $\mathfrak m_{\nu}$, then $R\subset \mathcal O_{\nu}$ and if $\mathfrak p=\mathfrak m_{\nu}\cap R$ then $\mbox{trdeg}_{\kappa(\mathfrak p)}\kappa(\nu)={\rm ht}(\mathfrak p)-1$, where $\kappa(\mathfrak p)$ is the residue field of $R_{\mathfrak p}$ and $\kappa(\nu)$ is the residue field of $\mathcal O_{\nu}$. If $R$ is a local ring and $\nu$ is a divisorial valuation of $R$ such that $m_R=m_{\nu}\cap R$, then $\nu$ is called an $m_R$-valuation.
 
 By \cite[Theorem 9.3.2]{HS}, the valuation ring of every divisorial valuation $\nu$ is Noetherian, hence is a  discrete valuation. 
 Suppose that  $R$ is an excellent local domain. Then a valuation $\nu$ of the quotient field $K$ of $R$ which is nonnegative on $R$ is a divisorial valuation of $R$ if and only if the valuation ring $\mathcal O_{\nu}$  is essentially of finite type over $R$ (\cite[Lemma 6.1]{CPS1} or \cite[Theorem 9.3.2]{HS}).
 
 Let $R$ be a Noetherian domain and $v$ be a discrete valuation of the quotient field of $R$ which is nonnegative on $R$. For $\lambda\in \RR_{\ge 0}$, define the valuation ideal of $R$
 $$
 I_R(v)_{\lambda}=\{f\in R\mid v(f)\ge \lambda\}=I(v)_{\lambda}\cap R.
 $$
 
A divisorial filtration of a Noetherian  domain $R$ is a filtration $\mathcal I=\{I_n\}$ where there exist divisorial valuations $v_1,\ldots,v_r$ of $R$ and 
$\lambda_1,\ldots,\lambda_r\in \RR_{\ge 0}$ such that 
$$
I_n=I_R(v_1)_{n\lambda_1}\cap \cdots\cap I_R(v_r)_{n\lambda_r}=I(v_1)_{n\lambda_1}\cap \cdots\cap I(v_r)_{n\lambda_r}\cap R.
$$
We now weaken our assumptions on $R$ and 
suppose that $R$ is an arbitrary Noetherian local ring. 

\begin{Definition} Let $R$ be a Noetherian  local ring and $P_1,\ldots,P_t$ be the minimal prime ideals of $R$. Let $\phi_i:R\rightarrow R/P_i$ be the natural projection  for $1\le i\le t$. Let $\{v_{ij}\}$ be divisorial valuations on $R/P_i$ and  $\lambda_{ij}\in \RR_{\ge 0}$. Let $\mathcal I=\{I_n\}$ where 
\begin{equation}\label{eq8}
I_n=\cap_{i=1}^t\cap_{j} \phi_i^{-1}(I_{R/P_i}(v_{ij})_{n\lambda_{ij}}).
\end{equation}
$\mathcal I$ is a graded filtration of $R$, which is called an $\RR$-divisorial filtration. If the $\lambda_{ij}$ are in $\QQ_{\ge 0}$ then $\mathcal I$ is called a $\QQ$-divisorial filtration and if the $\lambda_{ij}$ are in $\ZZ_{\ge 0}$ then $\mathcal I$ is called a divisorial filtration.
\end{Definition}

If $P$ is a prime ideal of a local ring $R$, $\phi:R\rightarrow R/P$ and $v$ is a valuation of the quotient field of $R/P$, then $w=v\circ\phi$ is a valuation of $R$, as defined in VI.3.1 \cite{Bou}. That is, by extending the value group $vR/P$ of $V$ by introducing the element $\infty$, which is larger than any element of $vR/P$, and letting $\Gamma_{\infty}=vR/P\cup\{\infty\}$, the mapping
$w:R\rightarrow \Gamma_{\infty}$ satisfies the properties
\begin{equation}\label{eq30}
\begin{array}{l}
w(xy)=w(x)+w(y)\mbox{ for }x,y\in R\\
w(x+y)\ge\min\{w(x),w(y)\}\mbox{ for }x,y\in R\\
w(1)=0\mbox{ and }w(P)=\infty
\end{array}
\end{equation}
With the notation of (\ref{eq8}), the mappings $v_{ij}\circ\phi_i$ are valuations of $R$, satisfying the conclusions of (\ref{eq30}).

%Let $R$ be a local ring and $w$ be a valuation of $R$. 
%For $\lambda\ge 0$ in $\Gamma_{\infty}$, let 
%$$
%I(w)_{\lambda} = \{x\in R\mid v(c)\ge \lambda\},
%$$
%an ideal in $R$.
%The center of $w$ on $R$ is the prime ideal $\mbox{center}_R(v)=\{x\in R\mid w(x)>0\}$.  

 %A representation of a divisorial filtration $\mathcal I$ on a normal  excellent local domain $R$ is a filtration $\mathcal I=\mathcal I(D)$ where $D$ is an effective $\RR$-Weil divisor on a normal $R$-scheme $X$ with a projective birational morphism $\pi:X\rightarrow \mbox{Spec}(R)$ such that $I_n=I(nD)=\Gamma(X,\mathcal O_X(-\lceil nD\rceil))$ for $n\ge 0$. 
%If $D=\sum\lambda_iE_i$ where $E_i$ are prime divisors on $X$, then $\lceil nD\rceil=\sum \lceil n\lambda_i\rceil E_i$ where $\lceil x\rceil$ is the round up of a real number $x$.
%$\mathcal I(D)$ is called a $\QQ$-divisorial filtration if $D$ is an effective $\QQ$-divisor and is called a $\ZZ$-divisorial filtration, or just a divisorial filtration if $D$ is a $\ZZ$-divisor. 

Associated to a discrete valuation ring $V$ is a unique canonical valuation $v_V$ defined by $v_V(x)=n$ if $x\in m_V^n\setminus m_V^{n+1}$ for $x\in V$ and where $m_V$ is the maximal ideal of $V$. If $w$ is any valuation whose valuation ring is $V$, so that $w$ is an equivalent valuation to $v$, we have that $w=cv$ where $c=w(u)$ for $u$ a generator of $m_V$ (a uniformizing parameter of $V$).

Now suppose that $R$ is a normal excellent local domain. Let 
\begin{equation}\label{N1}
\pi:X\rightarrow \mbox{Spec}(R)
\end{equation}
be a birational projective morphism such that $X$ is normal. Let
\begin{equation}\label{N2}
D=\sum\lambda_iE_i
\end{equation}
be an effective $\RR$-Weil divisor on $X$ (so that all $\lambda_i\in \RR_{\ge 0}$), where the $E_i$ are prime divisors on $X$. Define
$$
\lceil nD\rceil=\sum \lceil n\lambda_i\rceil E_i,
$$
where $\lceil x\rceil$ is the round up of a real number $x$. Then for all $n$,
$I(nD)=\Gamma(X,\mathcal O_X(-\lceil nD\rceil))$ is an ideal of $R$ and $\mathcal I=\{I(nD)\}$ is a divisorial filtration of $R$.

%Suppose that  $\mathcal I=\{I_n\}$ where $I_n=I_R(v_1)_{n\lambda_1}\cap\cdots\cap I_R(v_r)_{n\lambda_r}$ is a divisorial filtration of $R$ and $D=\sum_{i=1}^r \lambda_iE_i$ is an effective $\RR$-Weil divisor on a normal $R$-scheme $X$ and $E_i$ are prime divisors on $X$ such that $X$ is projective and birational over $\mbox{Spec}(R)$.

 Let $v_1,\ldots,v_r$ be the canonical valuations of the valuation rings $\mathcal O_{X,E_i}$. Then 
\begin{equation}\label{N3}
\Gamma(X,\mathcal O_X(-\lceil nD\rceil))=I_R(v_1)_{n\lambda_1}\cap \cdots\cap I_R(v_r)_{n\lambda_r}.
\end{equation}

Let $\mathcal I=\{I_n\}$ be a divisorial filtration of $R$. By \cite[Remark 6.6 to Lemma 5.5]{CPS1}, there exists $\pi:X\rightarrow \mbox{Spec}(R)$ and $D$ as in (\ref{N1}) and (\ref{N2}) such that $I_n=I(nD)$ for all $n\ge 0$.

A choice of $X$ and $D$ such that $I_n=I(nD)$ for all $n\ge 0$ is called a representation of the divisorial filtration $\mathcal I$. The choice of $X$ is not unique, and the choice of $D$ on a particular $X$ may not be unique. It is possible for a particular $\mathcal I$ to have representations (\ref{N3}) which are $\QQ$-divisorial and also have representations (\ref{N3}) which are not $\QQ$-divisorial.

%Suppose that $D$ is $\QQ$-Cartier, so there exists $c\in \ZZ_{>0}$ such that $cD$ is a Cartier divisor on $X$.
Let $K$ be the quotient field of $R$, $K\rightarrow L$ be a finite extension of fields  and $\{w_{ij}\}$ be the extensions of $v_i$ to $L$ for $1\le i\le r$. 
Let $S$ be the normalization of $R$ in $L$ and $Y$ be the normalization of $X$ in $L$, with induced morphism $\phi:Y\rightarrow X$. 
%Then  $\phi^*(cD)$ is a Cartier divisor on $Y$, so that $D'=\phi^*(D)$ is a $\QQ$-Cartier divisor on $Y$.  

Let $\{w_{ij}\}$ be the extensions of $v_i$ to $L$. Let $u_{i,j}$ be the canonical valuation of $\mathcal O_{w_{ij}}$, that is, a uniformizing parameter $t$ of $\mathcal O_{w_{i,j}}$ satisfies $u_{i,j}=1$. There exists a prime divisor $F_{ij}$ on $Y$ such that $\mathcal O_{w_{ij}}=\mathcal O_{F_{ij}}$. Write $w_{ij}=d_{ij}u_{ij}$.  Let $D'=\sum_{i,j}\frac{\lambda_i}{d_{ij}}F_{ij}$. Then

\begin{equation}\label{eq6}
I(D')_n=\cap_i\cap_jI(u_{ij})_{n\frac{\lambda_i}{d_{ij}}}\cap S=
\cap_i\cap_jI(w_{ij})_{n\lambda_i}\cap S
\end{equation}
for all $n\in \NN$.

\section{A generalized McAdam theorem}

In this section we prove a theorem (Theorem \ref{Theorem2})  which is the essential part of the proof of Theorem \ref{Theorem3} stated in the introduction.

\subsection{Resolutions and Alterations}

The proof of  \cite[Section 5]{SDCProc} shows the following.  

\begin{Theorem}(\cite[Section 5]{SDCProc})\label{Theorem1}
Let $R$ be a normal excellent local domain and let $\mathcal I=\{I_n\}$ be a $\QQ$-divisorial filtration on $R$. Suppose that there exists a projective birational morphism $\pi:X\rightarrow \mbox{Spec}(R)$ such that $X$ is nonsingular, and there exists an effective $\QQ$-divisor $D$ on $X$ such that $\mathcal I=\mathcal I(D)$.  If $m_R\in \mbox{Ass}(R/I_{m_0})$ for some $m_0\in \ZZ_{>0}$, then the analytic spread of $\mathcal I$ is $\ell(\mathcal I)=\dim R$. 
%Further, there exists $n_0\in \ZZ_{>0}$ such that $m_R\in \mbox{Ass}(R/I_n)$ if $n\ge n_0$.
\end{Theorem}

The following proposition is an  extension of of \cite[Section 6 ]{SDCProc}, with essentially the same proof.

\begin{Proposition}(\cite[Section 6]{SDCProc}\label{Prop1}
Let $R\rightarrow S$ be a finite extension of reduced rings where $R$ is  a Noetherian local ring of dimension $d$ and $S$ is a semi local ring. Let $\mathcal J=\{J_n\}$ be a graded filtration of $S$ and let $\mathcal I=\{I_n\}$ be the graded filtration of $R$ defined by $I_n=J_n\cap R$. For $q$ a prime ideal of $S$, let $\mathcal J_q$ be the graded filtration $\mathcal J_q=\{(J_n)_q\}$ of $S_q$. 
Suppose that  $m_R\in \mbox{Ass}_R(R/I_{m_0})$ for some $m_0>0$. 
Then there exists a maximal ideal $q$ of $S$ such that $qS_q\in \mbox{Ass}_{S_q}(S_q/(J_{m_0})_q)$. If we further have that  $R$ and $S$ are equidimensional of dimension $d$,
$\ell(\mathcal J_q)=d$ and there exists a nonzero divisor of $R$ in $M=\mbox{Ann}_R(S/R)$, then $\ell(\mathcal I)=d$.
\end{Proposition}

\begin{proof} Let $\{\mathfrak m_i$\} be the maximal ideals of $S$.
   Since we assume that $m_R\in \mbox{Ass}_R(R/I_{m_0})$, there exists $y\in R/I_{n_0}$ such that $m_R=\mbox{ann}_R(y)$. We have that $y\in S/J_n$ is nonzero since  $R/I_{m_0}\rightarrow S/J_{m_0}$ is an inclusion. Thus $\mbox{Ann}_S(y)\ne S$. Since maximal elements in the set of annihilators of elements of $S/J_{m_0}$ are prime ideals (by \cite[Theorem 6.1]{Ma}), there exists a prime ideal $q$ in $S$ which contains $\mbox{Ann}_S(y)$ and is the annihilator of an element $z$ of $S/J_{m_0}$. We have that $q$ contains $m_RS$ and 
   \begin{equation}\label{eqZ3}
   \sqrt{m_RS}=\cap \mathfrak m_i
   \end{equation}
    so $q$ is a maximal ideal of $S$ such that  $q\in\mbox{Ass}_S(S/J_{m_0})$, and thus $qS_q\in \mbox{Ass}_{S_q}(S_q/(J_{m_0})_q)$.   
     
    Further suppose that 
    \begin{equation}\label{eqZ2}
   \ell(\mathcal J_q)=\dim (S_{q})=\dim R
   \end{equation}
   and there exists a nonzero divisor of $R$ in $M=\mbox{Ann}_R(S/R)$

   Let $B=\oplus_{n\ge 0}J_n$, which is a graded ring. 
 We thus have by (\ref{eqZ1}), and (\ref{eqZ2})  that $\dim B/m_RB=\dim R$.
 Thus there is a chain of distinct prime ideals 
 $$
 C_0\subset C_1\subset C_2\subset \cdots\subset C_d
 $$
 in $B$ which contain $m_RB$, where $d=\dim R$.

 There is a natural inclusion of graded rings $R[\mathcal I]=\oplus_{n\ge 0}I_n\subset B=\oplus_{n\ge 0}J_n$. We will now show that $B$ is integral over $R[\mathcal I]$. It suffices to show that homogeneous elements of $B$ are integral over $R[\mathcal I]$.

 For $a\in \ZZ_{>0}$, let $R[\mathcal I]_a$ be the $a$-th truncation of $R[\mathcal I]$ and $B_a$ be the $a$-th truncation of $B$, so that $R[\mathcal I]_a$ is the subalgebra of $R[\mathcal I]$ generated by $\oplus_{n\le a}I_n$ and $B_a$ is the subalgebra of $B$ generated by $\oplus_{n\le a}J_n$.   Suppose that $f\in J_a$ for some $a$.
 Then $f\in B_a$. Let $0\ne x$ be a nonzero divisor of $R$ in $\mbox{Ann}_R(S/R)$. Then $xJ_n\subset I_n$ for all $n$ since $I_n=J_n\cap R$.
Thus $xB_a\subset R[\mathcal I]_a$, so $f^i\in \frac{1}{x}R[\mathcal I]_a$ for all $i\in \NN$, and so the algebra 
$R[\mathcal I]_a[f]\subset \frac{1}{x}R[\mathcal I]_a$. Since $\frac{1}{x}R[\mathcal I]_a$ is a finitely generated $R[\mathcal I]_a$-module and $R[\mathcal I]_a$ is a Noetherian ring, the ring $R[\mathcal I]_a[f]$ is a finitely generated $R[\mathcal I]_a$-module, so that $f$ is integral over $R[\mathcal I]_a$.

We have a chain of prime ideals 
$$
Q_0\subset Q_1\subset Q_2\subset \cdots\subset Q_d
$$
in $R[\mathcal I]$ where $Q_i:=C_i\cap R[\mathcal I]$. The $Q_i$ are all distinct since the $C_i$ are all distinct and $B$ is integral over $R[\mathcal I]$ (by \cite[Theorem A.6 (b)]{BH}). We have  that $m_RR[\mathcal I]\subset Q_0$, so that 
$\dim R[\mathcal I]/m_RR[\mathcal I]\ge d$. Since this is the maximum possible dimension of $R[\mathcal I]/m_RR[\mathcal I]$ by (\ref{eqZ1}), we have that $\ell(\mathcal I)=d$.

   \end{proof}

\begin{Lemma}\label{LemmaE} Suppose that $R$ is an excellent local domain and $v_1,\ldots,v_r$ are divisorial valuations in $R$. Then there exists a birational projective morphism $\pi:X\rightarrow \mbox{Spec}(R)$ such that $X$ is normal and there exist prime divisors $F_1,\ldots,F_r$ on $X$ such that $\mathcal O_{X,F_i}$ is the valuation ring of $v_i$ for $1\le i\le r$.
\end{Lemma}

\begin{proof} This follows from  \cite[Remark 6.6 to Lemma 6.5]{CPS1}.
\end{proof}

\begin{Definition}\label{AltDef}(\cite[Definition 2.20]{dJ})
Let $X$ be a Noetherian integral scheme. An Alteration $X'$ of $X$ is an integral scheme $X'$ with a morphism $\phi:X'\rightarrow X$ which is dominant, proper and such that for some open $U\subset X$, $\phi^{-1}(U)\rightarrow U$ is finite.
\end{Definition}

A scheme $X$ is nonsingular if $\mathcal O_{X,q}$ is a regular local ring for all $q\in X$. An alteration $\phi:X'\rightarrow X$ is regular if $X'$ is nonsingular. 

%\begin{Theorem}(\cite[Theorem 4.1]{dJ})\label{TheoremdJ} Let $X$ be a projective variety over a field $k$. Then there exists an alteration $\phi:X'\rightarrow X$ such that $X'$ is a projective $k$-variety and a nonsingular scheme.
%\end{Theorem}

\subsection{Associated primes and maximal analytic spread}\label{Secproof}

In this subsection we prove  the following theorem.
Theorem \ref{Theorem2} is proven for excellent equicharacteristic zero local domains in Theorem 1.3 \cite{SDCProc}.

\begin{Theorem}\label{Theorem2}
Let $R$ be an equidimensional  excellent local ring which has the property that if $P$ is a minimal prime of $R$ and 
$X\rightarrow \mbox{Spec}(R/P)$ is a birational projective morphism, with $X$ integral, then there exists a projective  regular alteration $X'\rightarrow X$.
Let $\mathcal I=\{I_n\}$ be a $\QQ$-divisorial filtration on $R$. If $m_R\in \mbox{Ass}(R/I_{m_0})$ for some $m_0\in \ZZ_{>0}$, then the analytic spread of $\mathcal I$ is $\ell(\mathcal I)=\dim R$. 
%Further, there exists $n_0\in \ZZ_{>0}$ such that $m_R\in \mbox{Ass}(R/I_n)$ if $n\ge n_0$.
\end{Theorem}

\begin{proof}

%Let $P_1,\ldots,P_t$ be the minimal prime ideals of $R$. 
%Then the reduced ring $S=R/\cap P_i$, the domains $S=R/P_i$ and the normalization $S$ of $R/P_i$ all have the property that if $X\rightarrow \mbox{Spec}(S)$ is a birational projective morphism, with $X$ integral, then there exists a regular alteration $X'\rightarrow X$.

Suppose that $m_R\in \mbox{Ass}_R(R/I_{m_0})$ for some $m_0$. Let $d=\dim R$.

\noindent{\bf Step 1.}  Suppose that $R$ is normal.
 By Lemma \ref{LemmaE}, there exists a birational projective morphism $Y_1\rightarrow\mbox{Spec}(R)$ such that $Y_1$ is normal and there exists an effective $\QQ$-Weil divisor $D_1$ on $Y_1$ such that  $\mathcal I=\mathcal I(D_1)$. 
 By the assumptions of the theorem, there exists a projective regular alteration $Z\rightarrow Y$.

 %The morphism $Y_1\rightarrow \mbox{Spec}(R)$ is the blow up of an ideal $J$ in $R$. $R$ is the local ring $\mathcal O_{V,q}$ of a (not necessarily closed) point $q$ on a projective $k$-variety $V$. Let $\mathcal J$ be an ideal sheaf on $V$ such that $\mathcal J_q=J$. Let $V_1\rightarrow V$ be the normalization of the blow up of $\mathcal J$.
 %Let $V_2\rightarrow V_1$ be a regular  alteration (whose existence follows from the assumptions of the theorem). This morphism is necessarily projective, since it is a morphism of projective varieties. 
 %Let $Z=V_2\times_V\mbox{Spec}(R)$. $Z$ is nonsingular since $R=\mathcal O_{V,q}$ so that $\mbox{Spec}(R)\rightarrow V$ is a morphism with geometrically regular fibers, so the induced morphism $Z\rightarrow Y_1$ 
 %is a projective regular alteration. 
 
Let $K$ be the quotient field of $R$ and let $L=\mathcal O_{Z,\eta}$ where $\eta$ is the generic point of $Z$. Let $S$ be the normalization of $R$ in $L$, with natural projective birational morphism $Z\rightarrow \mbox{Spec}(S)$. $S$ is a finitely generated $R$-module since $R$ is excellent. Let $q_1,\ldots,q_e$ be the distinct maximal ideals of $S$. Let $\mathcal I(R)=\mathcal I=\mathcal I(D_1)$. Let $Z_1\rightarrow Y_1$ be the normalization of $Y_1$ in $L$, with projective birational morphisms $Z\rightarrow Z_1\rightarrow \mbox{Spec}(S)$. 
Define $\mathcal I(S)$ from $\mathcal I(R)$ by  the construction of (\ref{eq5}). Define a $\QQ$-Weil divisor $D_1'$ on $Z_1$ from $D_1$ by the construction of (\ref{eq6}), so that $\mathcal I(S)=\mathcal I(D_1')$. Writing $D_1'=\sum b_{ij}F_{ij}$, let $D'$ be the $\QQ$-divisor $D'=\sum b_{ij}\overline F_{ij}$, where $\overline F_{ij}$
is the unique codimension 1 subvariety of $Z$ which dominates the codimension one subvariety $F_{ij}$ in the normal variety $Z_1$.  Then $\mathcal I(D')=\mathcal I(D_1')=\mathcal I(S)$.

Since $m_R\in \mbox{Ass}(R/I(R)_{m_0})$,  $I(R)_{m_0}:m_R^{\infty}\ne I(R)_{m_0}$ which implies that $I(S)_{m_0}:J(S)^{\infty}\ne I(S)_{m_0}$ by Lemma \ref{LemmaC}. Thus there exists some $q\in\{q_1,\ldots q_e\}$ such that $[I(S)_{m_0}:J(S)^{\infty}]_q\ne [I(S)_{m_0}]_q$.

Define the divisorial filtration $\mathcal I(S_q)=\{I(S_q)_n\}$ on the $d$-dimensional excellent normal local ring $S_q$ by the localization $I(S_q)_n=I(S)_n\otimes_SS_q$. We have that  $I(S_q)_{m_0}:m_{S_q}^{\infty}\ne I(S_q)_{m_0}$. The base change $Z_q=Z\times_SS_q \rightarrow \mbox{Spec}(S_q)$ is a birational projective morphism such that $Z_q$ is nonsingular. Further, $\mathcal I(S_q)=\mathcal I(D'|Z_q)$ for the $\QQ$-Cartier divisor $D'|Z_q$ on $Z_q$. By Theorem \ref{Theorem1}, $\ell(\mathcal I(S_q))=\dim S=d$.
$S[\mathcal I(S_q)]=S[\mathcal I(S)]\otimes_SS_q$ and $\dim S_q[\mathcal I(S_q)]/q_qS_q[\mathcal I(S_q)]
=\ell(\mathcal I(S_q))=d$  implies there exists a chain of distinct prime ideals $P_0\subset P_1\subset \cdots \subset P_d$ in $S[\mathcal I(S)]$ such that $q\subset P_0$. From the inclusion $R[\mathcal I(R)]\subset S[\mathcal I(S)]$, we have a chain of prime ideals $Q_0=P_0\cap R[\mathcal I(R)]\subset Q_1=P_1\cap R[\mathcal I(R)]\subset \cdots \subset Q_d=P_d\cap R[\mathcal I(R)]$.
Now $R[\mathcal I(R)]\rightarrow S[\mathcal I(S)]$ is an integral extension by Proposition \ref{PropD}. Thus the prime ideals $Q_0\subset Q_1\subset \cdots \subset Q_d$ are all distinct by  \cite[Theorem A.6]{BH}. Since these ideals all contain $q\cap R=m_R$, we have that $\ell(\mathcal I(R))\ge d$. Since we have the upper bound $\ell(\mathcal I(R))\le d$ of (\ref{eqZ1}),
we have that $\ell(\mathcal I(R))=d$.
This completes the proof of Theorem \ref{Theorem2} in the case that $R$ is normal.

\noindent{\bf Step 2.}
Now suppose that $R$ is a domain, which is not normal.
There exist  divisorial valuations 
   $\nu_1,\ldots,\nu_t$  of $R$ and $a_1,\ldots, a_t\in \QQ_{>0}$ such that $\mathcal I=\{I_n\}$ where 
   $I_n=I_R(\nu_1)_{a_1n}\cap \cdots \cap I_R(\nu_t)_{a_tn}$ for $n\ge 0$.   
   Let $S$ be the normalization of $R$ in the quotient field of $R$. 
   %Let $\mathfrak m_1,\ldots,\mathfrak m_u$ be the maximal ideals of $S$.
   Let $J(\nu_i)_m=\{f\in S\mid \nu_i(f)\ge m\}$.  
   %For each $i$, there exists $\sigma(i)$ with $0\le \sigma(i)\le u$ such that  $J(\nu_i)_1=\mathfrak m_{\sigma(i)}$ if $\sigma)i)\ge 1$ and $J(\nu)i)_1$ is not a maximal ideal of $S$ if $\sigma(i)=0$. . That is, $\nu_i$ is an $\mathfrak m_{\sigma(i)}$-valuation.
   For $n\in \NN$, let 
   $$
   J_n=J(\nu_1)_{a_1n}\cap \cdots \cap J(\nu_t)_{a_tn}
   $$
   so that $J_n\cap R=I_n$. Let $\mathcal J=\{J_n\}$.   
   By Proposition \ref{Prop1},  there exists a maximal ideal $q$ of $S$ such that $qS_q\in \mbox{Ass}_{S_q}(S_q/(J_{m_0})_q)$.   
Now $\mathcal J_q$ is a $\QQ$-divisorial filtration of the $d$-dimensional normal local ring $S_q$. A birational projective morphism $X\rightarrow \mbox{Spec}(S)$ induces a birational projective morphism $X\rightarrow \mbox{Spec}(R)$, so 
 $\ell(\mathcal J_q)=d$ by Step 1. Let $0\ne x\in R$ be in the conductor of $S$ over $R$, so that $x\in \mbox{Ann}_R(S/R)$ is a nonzero divisor on $R$. Then by Proposition \ref{Prop1}, we have that
$\ell(\mathcal I)=d$.

\noindent{\bf Step 3.} 
Suppose  that $R$ is reduced and equidimensional and $\mathcal I$ has an expression (\ref{eq8}).
 Let $S=\oplus_{i=1}^tR/P_i$ where $P_i$, $1\le i\le t$ are the minimal primes of $R$. Let $\mathcal J=\{J_n\}$ where 
 $$
 J_n=\oplus_{i=1}^t[\cap_jI_{R/P_i}(v_{ij})_{n\lambda_{ij}}].
 $$
 We have that $J_n\cap R=I_n$. By Proposition \ref{Prop1}, there exists a maximal ideal $q$ of $S$ such that $qS_q\in \mbox{Ass}_{S_q}(S_q/(J_{m_0})_q)$.
 
 Let $M=S/R$, $A=\mbox{Ann}_RM$. Localizing the exact sequence of $R$-modules 
 $$
 0\rightarrow R\rightarrow \oplus R/P_i\rightarrow M\rightarrow 0
 $$
 at $P_i$, we see that $M_{P_i}=0$ for all $i$. Thus $A_{P_i}=R_{P_i}$ for all $i$ so that $A\not\subset P_i$ for all $i$, and so $A\not\subset \cup P_i$ by prime avoidance. Thus there exists $x\in A\setminus \cup P_i$. Since $R$ is reduced, $x$ is a nonzero divisor of $R$. Now $S_q\cong (R/P_i)_{m_{R/P_i}}$ for some $i$, and $\mathcal J_q$ is a $\QQ$-divisorial filtration on the $d$-dimensional local domain $S_q$, so that $\ell(\mathcal J_q)=d$ by Step 2. Thus $\ell(\mathcal I)=d$ by Proposition \ref{Prop1} (this is where the assumption that $R$ is equidimensional is used).
 
 \noindent{\bf Step 4.}
  Now suppose that $R$ is an arbitrary equidimensional local ring. Let $\mathcal I=\{I_n\}$ have an expression (\ref{eq8}). Let $\tilde R=R/\cap_{i=1}^tP_i$, which is a reduced local ring, with natural surjection $R\rightarrow \tilde R$. Since $P_i$ is contained in the kernel of $\phi_i$ for all $i$, we have that $\cap_{i=1}^tP_i$ is contained in 
 $\phi^{-1}(I_{R/P_i}(v_{ij})_{n\lambda_{i,j}})$ for all $i,j$, and so $\cap_{i=1}^tP_i$ is contained in $I_n$ for all $n$.  Let $\alpha_i:\tilde R \rightarrow R/P_i$   be the natural projection   for $1\le i\le t$. Let $\tilde{\mathcal I}$ be the $\QQ$-divisorial filtration on $\tilde R$ defined by $\tilde I_n=\cap_{i=1}^t\cap_{i,j} \alpha_i^{-1}(I_{R/P_i}(v_{ij})_{n\lambda_{ij}})$. We have that $R/I_n\cong \tilde R/\tilde I_n$ for all $n$ and $m_{\tilde R}=m_R\tilde R$. Thus $m_R$ an associated prime of $R/I_n$ implies $m_{\tilde R}$ is an associated prime of $\tilde R/\tilde I_n$, so  $\ell(\tilde{\mathcal I})=\dim \tilde R=\dim R$ by Step 3. Since there is a natural surjection 
 $$
 R[\mathcal I]/m_RR[\mathcal I]\rightarrow \tilde R[\tilde{\mathcal I}]/m_{\tilde R}\tilde R[\tilde{\mathcal I}],
 $$
 we have that $\ell(\mathcal I)\ge \dim R$, so by (\ref{eqZ1}), we have that $\ell(\mathcal I)=\dim R$ and the theorem holds for $R$.
 
 \end{proof}

\section{A divisorial valuation whose first difference function doesn't have a limit}\label{SecExample}

\subsection{Boundedness of the difference function}
 In this subsection we prove the following lemma.
 
 \begin{Lemma}\label{Lemmadiff} Suppose that $R$ is a $d$-dimensional normal local ring which is essentially of finite type over a field and $\mathcal I=\{I_n\}$ is a $\QQ$-divisorial filtration of $m_R$-primary ideals on $R$. Then there exists a constant $c$ such that
 $$
 \lambda_R(I_n/I_{n+1})\le cn^{d-1}
 $$
 for all $n$. Thus 
 $$
 \limsup_{n\rightarrow\infty}\frac{\lambda_R(I_n/I_{n+1})}{n^{d-1}}
 $$
 is finite. 
 \end{Lemma}

\begin{proof} There exists a field $\tilde k$ such that $\tilde k\subset R$ and $[R/m_R:\tilde k]<\infty.$
Let $\mathcal I(R)=\mathcal I$. The proof of Step 1 of Theorem \ref{Theorem2} in Section \ref{Secproof} and \cite[Theorem 4.1]{dJ}, shows that there exists a finite field extension $L$ of the quotient field $K$ of $R$ such that if $S$ is the normalization of $R$ in $L$ then there exists a resolution of singularities $Z\rightarrow \mbox{Spec}(S)$ and an effective $\QQ$-divisor $D'$ on $Z$ such that $\mathcal I(S)=\mathcal I(D')$  and $I(S)_n\cap R=I(R)_n=I_n$ for all $n$. Thus we have inclusions 
\begin{equation}\label{eqN1}
I_n/I_{n+1}\subset I(S)_n/I(S)_{n+1}
\end{equation}
for all $n$. There exists a positive integer $a$ such that $aD'$ is a $\ZZ$-divisor. Write $n=m(n)a+b(n)$n where $0\le b(n)<a$. Let $F_n=\lceil nD'\rceil-m(n)aD'$. The $\{F_n\}$ are effective divisors such that set of all $F_n$ is a finite set.  Let
$$
G_n=\lceil (n+1)D'\rceil -\lceil nD'\rceil.
$$ 
Again, the set of all $G_n$ is finite. We have short exact sequences
$$
0\rightarrow \mathcal O_Z(-G_n)\rightarrow \mathcal O_Z\rightarrow \mathcal O_{G_n}\rightarrow 0
$$
yeilding short exact sequences
$$
0\rightarrow \mathcal O_Z(-\lceil (n+1)D'\rceil)\rightarrow \mathcal O_Z(-\lceil nD'\rceil)\rightarrow \mathcal O_{G_n}\otimes\mathcal O_Z(-m(n)aD'-F_n)
\rightarrow 0.
$$
Taking global sections, we see that 
\begin{equation}\label{eqN2}
\begin{array}{lll}
I(S)_n/I(S)_{n+1}&=&\Gamma(Z,\mathcal O_Z(-\lceil (n+1)D'\rceil))/\Gamma(Z,\mathcal O_Z(-\lceil nD'\rceil))\\
&\subset& \Gamma(Z,\mathcal O_{G_n}\otimes_{\mathcal O_Z}\mathcal O_Z(-m(n)aD'-F_n)).
\end{array}
\end{equation}
 Let $A$ be  an ample divisor on the projective $R$-scheme  $Z$ such that $-D'<A$.
 Let $F'\in \{F_n\}$ and $G'\in \{G_n\}$.  $G'$ is a projective $\tilde k$-scheme of dimension $d-1$. For a coherent sheaf $\mathcal F$ on $G'$ let $h^0(\mathcal F)=\dim_{\tilde k}\Gamma(G',\mathcal F)$. Then 
 $$
 h^0(\mathcal F)=[R/m_R:\tilde k]\ell_R(\Gamma(G',\mathcal F)).
 $$
 We have that
 $$
 h^0(\mathcal O_{G'}\otimes\mathcal O_Z(-laD-F'))\le h^0(\mathcal O_{G'}\otimes\mathcal O_Z(lA-F'))
 =\chi(\mathcal O_{G'}\otimes\mathcal O_Z(lA-F'))
 $$
 for $l\gg 0$ by Proposition III.5.3 \cite{H} or Theorem 17.8 \cite{AG}. The Euler characteristic $\chi(\mathcal O_{G'}\otimes\mathcal O_Z(lA-F'))$ is a polynomial in $l$ of degree $\le \dim G'=d-1$ by \cite{Sn}, \cite{Kl} or Theorem 19.1 \cite{AG}.
 Thus there exists a constant $c'$ such that
 \begin{equation}\label{eqN3}
 h^0(\mathcal O_{G_n}\otimes\mathcal O_Z(-m(n)aD-F_n))\le c'n^{d-1}
 \end{equation}
 for all positive $n$, and so,
 there exists a positive constant $c$ such that 
 $$
 \lambda_R(I_n/I_{n+1})\le cn^{d-1}
 $$
  for all $n$ by (\ref{eqN1}), (\ref{eqN2}) and (\ref{eqN3}).

\end{proof}

\subsection{Beatty Sequences}\label{SecBeatty}

For $x\in \RR$, $\lfloor x\rfloor$ will denote the round down of  $x$ and $\lceil x\rceil$ will denote the round up of  $x$.
The fractional part of $x$ is $\{x\}=x-\lfloor x\rfloor$.

Let $\alpha\in \RR$. The Beatty sequence of $\alpha$ is the sequence $\sigma_{\alpha}(n)=\lfloor \alpha(n+1)\rfloor - \lfloor \alpha n\rfloor$ for $n\in \ZZ_{>0}$.
We remark that if $\alpha$ is irrational, then
$\sigma_{\alpha}(n)=\lceil \alpha(n+1)\rceil-\lceil \alpha n\rceil$ for all $n\in \ZZ_{>0}$. We have the following classical result.

\begin{Proposition}\label{Prop1*} Suppose that $\alpha$ is an irrational number. Then there are exactly two values of $\sigma_{\alpha}(n)$ for $n\in \ZZ_{>0}$; they are $\lfloor \alpha\rfloor$ and $\lceil \alpha\rceil$. Each of these two values of $\sigma_{\alpha}(n)$ occur for infinitely many values of $n\in \ZZ_{>0}$.
\end{Proposition}

We will give a  short proof of Proposition \ref{Prop1*} in this subsection.

\begin{Definition} Let $x_n\in \RR$ be a sequence. The sequence $x_n$ is uniformly distributed Mod 1 if for every $a,b$ with $0\le a<b\le 1$, we have that
$$
\lim_{n\rightarrow\infty}\frac{1}{n}{\rm card}\{0\le j\le n-1\mid \{x_j\}\in [a,b]\}=b-a.
$$
\end{Definition}

A proof of the following theorem is given in formula (1.1) of page 9 of \cite{Pa}.

\begin{Theorem}\label{TheoremB1} Let $\alpha\in \RR$ be irrational. Then the sequence $n\alpha$ for $n\in \NN$ is uniformly distributed mod 1.
\end{Theorem}

We now give a proof of Proposition \ref{Prop1*},
 following the note \cite{Co}.

We have that $\alpha(n+1)-1<\lfloor\alpha(n+1)\rfloor<\alpha(n+1)$ and $\alpha n-1<\lfloor \alpha n\rfloor< \alpha n$. Thus
$$
\alpha(n+1)-1-\alpha n<\lfloor\alpha(n+1)\rfloor-\lfloor\alpha n\rfloor<\alpha(n+1)-(\alpha n-1)
$$
and so
$$
\alpha-1<\lfloor\alpha(n+1)\rfloor-\lfloor\alpha n\rfloor<\alpha+1.
$$
Since $\lfloor \alpha(n+1)\rfloor-\lfloor\alpha n\rfloor$ is an integer and $\alpha$ is not, 
$\lfloor\alpha(n+1)\rfloor-\lfloor\alpha n\rfloor$ is equal to $\lfloor  \alpha\rfloor$ or $\lceil \alpha\rceil$.
We have that
\begin{equation}\label{eqB1}
\lfloor\alpha(n+1)\rfloor-\lfloor\alpha n\rfloor
=\alpha(n+1)-\{\alpha(n+1)\}-(\alpha n-\{\alpha n\})
=\alpha+\{\alpha n\}-\{\alpha n+\alpha\}.
\end{equation}
For $x,y\in \RR$, we have that
$$
\{x+y\}=\left\{\begin{array}{l}
\{x\}+\{y\}\mbox{ if }\{x\}+\{y\}<1\\
\{x\}+\{y\}-1\mbox{ if }\{x\}+\{y\}\ge1.
\end{array}\right.
$$
By Theorem \ref{TheoremB1}, choosing $a,b$ such that 
$1-\{\alpha\}<a<b<1$, we find that there are infinitely many $n\in \NN$ such that $\{\alpha n\}+\{\alpha\}>1$. Hence, 
$$
\begin{array}{lll}
\sigma_{\alpha}(n)&=&\lfloor\alpha(n+1)\rfloor-\lfloor \alpha n\rfloor
=\alpha+\{\alpha n\}-\{\alpha n+\alpha\}=\alpha+\{\alpha n\}-(\{\alpha n\}+\{\alpha\}-1)\\
&=&\alpha-\{\alpha\}+1=\lceil\alpha\rceil.
\end{array}
$$
Hence $\lceil\alpha\rceil$ appears infinitely many times in the sequence $\sigma_{\alpha}(n)$.

Again by Theorem \ref{TheoremB1}, choosing $a,b$ so that $0<a<b<1-\{\alpha\}$, we find there are infinitely many $n\in \NN$ such that $\{\alpha n\}+\{\alpha\}<1$. Continuing as in the above paragraph, we find that $\lfloor\alpha\rfloor$ occurs infinitely many times in the sequence $\sigma_{\alpha}(n)$.

\subsection{Geometric preliminaries} Let $V$ be a projective variety over a field $k$ and $\mathcal F$ be a coherent sheaf on $V$. We will write $h^i(\mathcal F)=\dim_kH^i(V,\mathcal F)$.
The Euler characteristic of $\mathcal F$ is 
$$
\chi(\mathcal F)=\sum_{i=0}^{\dim\mathcal F}(-1)^ih^i(\mathcal F).
$$

Similar statements to the following Lemma \ref{LemmaVan} are well known. For the readers convenience, we give a proof.

\begin{Lemma}\label{LemmaVan}
Suppose that $T$ is a nonsingular projective surface over a field of characteristic zero and $\{D_n\}_{n\in\ZZ_{>0}}$ is a family of ample divisors on $T$, having the property that if $C$ is a curve on $T$, then there exists a constant $\gamma(C)$ such that $(C\cdot D_n)\le \gamma(C)n$ for all $n\in\ZZ_{>0}$. Suppose that $G$ is a divisor on $T$. Then there exists a constant $c$ such that
$$
h^1(\mathcal O_T(D_n+G))\le cn\mbox{ and }h^2(\mathcal O_T(D_n+G))\le c\mbox{ for all $n\in\ZZ_{>0}$.}
$$
\end{Lemma}

\begin{proof}
By Kodaira vanishing, 
\begin{equation}\label{eqVan2}
H^i(T,\mathcal O_T(D_n+K_T))=0\mbox{ for all $n>0$ and $i>0$}
\end{equation}
 where $K_T$ is a canonical divisor on $T$.

Suppose that $Z$ is an integral nonsingular curve on $T$. From the short exact sequence $0\rightarrow \mathcal O_T(-Z)\rightarrow \mathcal O_T\rightarrow \mathcal O_Z\rightarrow 0$, we have exact sequences
$$
\begin{array}{l}
H^1(T,\mathcal O_T(D_n+K_T))\rightarrow H^1(T,\mathcal O_T(D_n+K_T+Z))\rightarrow H^1(Z,\mathcal O_T(D_n+K_T+Z)\otimes\mathcal O_Z))\\\rightarrow H^2(T,\mathcal O_T(D_n+K_T))\rightarrow H^2(T,\mathcal O_T(D_n+K_T+Z))\rightarrow 0.
\end{array}
$$
Since $((D_n+K_T+Z)\cdot Z)$ is bounded from below (as the $D_n$ are ample) we have that $h^1(\mathcal O_T(D_n+K_T+Z)\otimes\mathcal O_Z))$ is bounded by Lemma 3.1 \cite{C1}. Thus using equation (\ref{eqVan2}),
there exists a constant $c$ such that
\begin{equation}\label{eqVan1}
h^1(\mathcal O_T(D_n+K_T+Z))\le c\mbox{ and }h^2(\mathcal O_T(D_n+K_T+Z))=0\mbox{ for all }n>0.
\end{equation}

Now suppose that $G'$ is a divisor on $T$. By Bertini's theorem, we have a linear equivalence $G'\sim A-B$ where $A$ and $B$ are ample effective divisors and $A$ and $B$ are integral nonsingular curves on $S$. From the short exact sequence
$$
0\rightarrow \mathcal O_T(-B)\rightarrow \mathcal O_T\rightarrow \mathcal O_B\rightarrow 0,
$$
 we have exact sequences 
 $$
 \begin{array}{l}
 H^0(B,\mathcal O_T(D_n+K_T+A)\otimes\mathcal O_B)\rightarrow H^1(T,\mathcal O_T(D_n+K_T+G'))\\
 \rightarrow H^1(T,\mathcal O_T(D_n+K_T+A))
 \rightarrow H^1(B,\mathcal O_T(D_n+K_T+A)\otimes\mathcal O_B)\\
 \rightarrow H^2(T,\mathcal O_T(D_n+K_T+G'))
 \rightarrow H^2(T,\mathcal O_T(D_n+K_T+A))\rightarrow 0.
  \end{array}
 $$
 
 By the Riemann Roch theorem on a nonsingular curve (formula (6)  and Lemma 3.1 \cite{C},  Section 7.3 \cite{L} or Section IV.1 of \cite{H} over an algebraically closed field), Lemma 3.1 \cite{C1} and our assumption that $\deg(\mathcal O_T(D_n)\otimes\mathcal O_B)\le\gamma(B)n$, we have that there exists a constant $c'$ such that
 $h^0(B,\mathcal O_T(D_n+K_T+A)\otimes\mathcal O_B)\le c'n$ for all $n$. Now using formula (\ref{eqVan1}), we see that there is a constant $c_1$ such that
 \begin{equation}\label{eqVan4}
 h^1(\mathcal O_T(D_n+K_T+G')\le c_1n\mbox{ and }h^2(\mathcal O_T(D_n+K_T+G'))\le c_1
 \end{equation}
 for all $n\in \NN$.
 
 Let $G$ be a divisor on $T$. Take $G'=G-K_T$ in (\ref{eqVan4}), and the conclusions of the lemma follow.
 \end{proof}

\subsection{The construction}

We recall the construction of Section 4 of \cite{C}.  A commutative diagram of projective varieties
$$
\begin{array}{rll}
Y\\
\tau\downarrow&\searrow\\
X&\stackrel{\lambda}{\rightarrow}&\overline Z\\
\pi\downarrow\\
S
\end{array}
$$
is constructed, where $X$ is a 3-fold which is a projective bundle over an Abelian surface $S$ (which we assume is over an algebraically closed field $k$ of characteristic zero) and $\tau$ is the blowup of a nonsingular curve $C$ lying on the zero section $S_0$ of $\pi$, with exceptional divisor $F$. $\tau$ induces a morphism $\overline\tau:F\rightarrow C$, realizing $F$ as   a ruled surface over $C$. Let $\overline S$ be the strict transform of $S_0$, which is isomorphic to $S$. Let $\overline Z$ be the normal projective 3-fold which is obtained by contracting $S_0$ (it is shown in \cite{C} that there is a morphism $X\rightarrow \overline Z$ with this property). 
Let $q$ be the closed (singular) point of $\overline Z$ which is the image of $S_0$. In  \cite{C} it is shown that $\lambda:X\rightarrow \overline Z$ is the blowup of $\lambda_*\mathcal O_X(-mS_0)$ for some large $m$, which is  an ideal sheaf whose cosupport is $\lambda(S_0)=q$. Since $\tau:Y\rightarrow X$ is the blowup of an ideal sheaf whose cosupport $C$ is contained in $S_0$, we have that $Y\rightarrow \overline Z$ is the blowup of an ideal sheaf whose cosupport is $q$.

Let $R=\mathcal O_{\overline Z,q}$ and let $U=Y\times_{\overline Z}\mbox{Spec}(R)$ with induced projective morphism $U\rightarrow \mbox{Spec}(R)$. 
By the above, $U$ is the blowup of an $m_R$-primary ideal $I$ of $R$.

\subsection{The example}\label{SecExample1}

By the formula before Lemma 2.2 \cite{C} and Theorem 4.1 \cite{C}, we have that
$$
I(nF):=\Gamma(U,\mathcal O_U(-nF))=\Gamma(U,\mathcal O_U(-\lceil n\frac{3}{9-\sqrt{3}}\rceil \overline S-nF)).
$$
Let $D_n=\lceil n\frac{3}{9-\sqrt{3}}\rceil \overline S+nF$ for $n\in \ZZ_{>0}$. The closure of the ample cone of $\overline S$  is equal to the closure of the nef cone of $\overline S$; that is, $\overline{\rm Amp}(\overline S) =\overline{\rm Nef}(\overline S)$. It is  shown before formula (14) on page 12 of \cite{C}
that 
\begin{equation}\label{eqEx1}
\mathcal O_U(-D_n)\otimes\mathcal O_{\overline S}\mbox{ is ample}. 
\end{equation}

Let $C_0$ be the zero section of $\overline\tau:F\rightarrow C$. The closure of the effective cone of $F$ is $\overline{\rm Eff}(F)=\RR_{\ge 0}C_0+\RR_{\ge 0}f$ by a formula on page 15 \cite{C}, where $f$ is the fiber by $\overline\tau$ of a closed point of $C$. It is shown on page 14 \cite{C} that the class of $\mathcal O_U(F)\otimes\mathcal O_F$ in $({\rm Pic}(F)/\equiv)\otimes\RR$ is that of $\mathcal O_U(-C_0-108f)$. Thus the class of 
$\mathcal O_U(-D_n)\otimes \mathcal O_F$ in $({\rm Pic}(F)/\equiv)\otimes\RR$ is that of
$\mathcal O_U((n-\lceil n\frac{3}{9-\sqrt{3}})C_0+n108 f)$.

On page 14 \cite{C}, the intersection products 
$$
(C_0^2)_F=-162, (C_0\cdot f)_F=1 \mbox{ and }(f^2)_F=0
$$
are computed. We may thus compute that
$$
((\mathcal O_U(-D_n)\otimes\mathcal O_F)\cdot C_0)>0\mbox{ and }(\mathcal O_U(-D_n)\otimes\mathcal O_F)\cdot f)>0
$$
for $n\gg 0$,
so that $((\mathcal O_U(-D_n)\otimes\mathcal O_F)^2)>0$. Thus there exists $n_0$ such that
\begin{equation}\label{eqEx3}
\mathcal O_U(-D_n)\otimes\mathcal O_F\mbox{ is ample }
\end{equation}
for $n\ge n_0$ 
by the Nakai-Moishezon criterion (Theorem V.1.10 \cite{H}).

\begin{Lemma}\label{Lemmabound} There exists a constant $c'$ such that 
$$
\lambda_R(H^1(U,\mathcal O_U(-D_n))\le c'n\mbox{ and }\lambda_R(H^2(U,\mathcal O_U(-D_n))\le c'\mbox{ for all $n\ge 0$}.
$$
\end{Lemma}

\begin{proof} As we observed at the beginning of this section, $U$ is the blowup of an $m_R$-primary ideal $I$, so there exists an effective  divisor $A=a\overline S+bF$ such that 
$\mathcal O_U(-A)=I\mathcal O_U$ and so 
$-A$ is ample on $U$. There exists a sequence of surfaces $T_1, T_2,\ldots,T_r$ where $r=a+b$ such that each $T_i$ is one of $\overline S$ or $F$ and $A=T_1+\cdots+T_r$. We have short exact sequences
\begin{equation}\label{eqVan8}
\begin{array}{l}
0\rightarrow O_U(-T_1)\otimes \mathcal O_{T_2}\rightarrow \mathcal O_{T_1+T_2}\rightarrow \mathcal O_{T_1}\rightarrow 0\\
\,\,\,\,\,\,\,\,\,\,\,\,\,\,\,\vdots\\
0\rightarrow \mathcal O_U(-(T_1+\cdots+T_{r-1}))\otimes\mathcal O_{T_r}\rightarrow \mathcal O_A\rightarrow \mathcal O_{T_1+\cdots+T_{r-1}}\rightarrow 0\\
0\rightarrow \mathcal O_U(-A)\otimes\mathcal O_{T_1}\rightarrow \mathcal O_{A+T_1}\rightarrow \mathcal O_A\rightarrow 0\\
0\rightarrow \mathcal O_U(-A-T_1)\otimes \mathcal O_{T_2}\rightarrow \mathcal O_{A+T_1+T_2}\rightarrow \mathcal O_{A+T_1}\rightarrow 0\\
\,\,\,\,\,\,\,\,\,\,\,\,\,\,\,\vdots\\
0\rightarrow \mathcal O_U(-A-(T_1+\cdots+T_{r-1}))\otimes\mathcal O_{T_r}\rightarrow \mathcal O_{2A}\rightarrow \mathcal O_{A+T_1+\cdots+T_{r-1}}\rightarrow 0\\
\,\,\,\,\,\,\,\,\,\,\,\,\,\,\,\vdots\\
\end{array}
\end{equation}
There exists $t_0$ such that $t\ge t_0$ implies $\mathcal O_U(-tA-(T_1+\cdots+T_{i-1}))\otimes\mathcal O_{T_i}(-K_{T_i})$ is ample on $T_i$, for $1\le i\le r$, where $K_{T_i}$ is a canonical divisor on $T_i$. Since $\mathcal O_U(-D_n)\otimes \mathcal O_{T_i}$ are ample for all $i$ and $n\ge n_0$ by 
(\ref{eqEx1}) and (\ref{eqEx3}),
$$
H^i(T_i,\mathcal O_T(-D_n-tA-(T_1+\cdots+T_{i-1}))\otimes\mathcal O_{T_i})=0
$$
 for all $n\ge n_0$, $i$ and $t\ge t_0$ by Kodaira vanishing. From the exact sequences of (\ref{eqVan8}), we see that 
\begin{equation}\label{eqVan6}
H^i(\mathcal O_U(-D_n)\otimes\mathcal O_{mA+T_1+\cdots+T_i})\cong H^i(U,\mathcal O_U(-D_n)\otimes\mathcal O_{t_0A})
\end{equation}
for $n\ge n_0$, $m\ge t_0$ and $1\le i\le r-1$. We may now replace $A$ with $t_0A$ and replace $r$ with $u=t_0r$ and consider the short exact sequences
\begin{equation}\label{eqVan7}
\begin{array}{l}
0\rightarrow \mathcal O_T(-T_1)\otimes\mathcal O_{T_2}\rightarrow \mathcal O_{T_1+T_2}\rightarrow \mathcal O_{T_1}\rightarrow 0\\
\,\,\,\,\,\,\,\,\,\,\,\,\,\,\,\vdots\\
0\rightarrow \mathcal O_T(-(T_1+\cdots +T_{u-1}))\otimes\mathcal O_{T_u}\rightarrow \mathcal O_A\rightarrow \mathcal O_{T_1+\cdots+T_{u-1}}\rightarrow 0.
\end{array}
\end{equation}
By (\ref{eqEx1}), (\ref{eqEx3}) and Lemma \ref{LemmaVan}, there exists a positive constant $c_2$ such that
$$
h^1(\mathcal O_T(-D_n)\otimes\mathcal O_{T_1})\le c_2n,\,\, h^1(\mathcal O_U(-D_n-(T_1+\cdots+T_{i-1}))\otimes \mathcal O_{T_i}))\le c_2n,
$$
$$
h^2(\mathcal O_T(-D_n)\otimes\mathcal O_{T_1})\le c_2,\,\, h^2(\mathcal O_T(-D_n-(T_1+\cdots+T_{i-1}))\otimes \mathcal O_{T_i}))\le c_2
$$
for $2\le i\le u$ and all $n\ge 0$. 

Thus for $n\ge n_0$,
$$
\begin{array}{l}
\lambda_R(\lim_{\leftarrow}H^1(U,\mathcal O_U(-D_n)\otimes\mathcal O_U/\mathcal O_U(-mA)))\\
=\dim_k(\lim_{\leftarrow}H^1(U,\mathcal O_U(-D_n)\otimes\mathcal O_U/\mathcal O_U(-mA)))\\
\le uc_2n
\end{array}
$$
and
$$
\begin{array}{l}
\lambda_R(\lim_{\leftarrow}H^2(U,\mathcal O_U(-D_n)\otimes\mathcal O_U/\mathcal O_U(-mA)))\\
=\dim_k(\lim_{\leftarrow}H^2(U,\mathcal O_U(-D_n)\otimes\mathcal O_U/\mathcal O_U(-mA)))\\
\le uc_2,
\end{array}
$$ 
 where the inverse limits are over $m\in\ZZ_{>0}$. Now $I$ is $m_R$-primary and $I^m\mathcal O_U=\mathcal O_U(-mA)$, so that 
$$
\lim_{\leftarrow} H^i(U,\mathcal O_U(-D_n)\otimes_R R/m_R^m)\cong
\lim_{\leftarrow}H^i(U,\mathcal O_U(-D_n)\otimes_{\mathcal O_U}\mathcal O_U/\mathcal O_U(-mA))
$$
where the limits are over $m\in\ZZ_{>0}$.
Since the $R$-modules $H^i(U,\mathcal O_U(-D_n))$ are supported on $m_R$, The $m_R$-adic completions of these modules are isomorphic to themselves. By the theorem on formal functions, Theorem III.1.1 \cite{H}, 
$$
H^i(U,\mathcal O_U(-D_n))\cong \lim_{\leftarrow} H^i(U,\mathcal O_U(-D_n)\otimes_RR/m_R^n)
$$
for $n\ge n_0$ and $i=1$ and 2. Now choose $c_3$ so that $\lambda_RH^i(U,\mathcal O_U(-D_n))\le c_3$ for $n\ge n_0$ and $i=1,2$. Let $c'=uc_2+c_2$, so the conclusions of the lemma hold.

\end{proof}

We now establish that
\begin{equation}\label{form1}
\lambda_R(R/I(nF))=\chi(\mathcal O_{D_n})+\tau(n)\mbox{ for $n\in \NN$}
\end{equation}
where $\tau(n)$ is a function such that there exists a constant $c$ such that $|\tau(n)|\le cn$ for all $n\in \NN$.
We start with the short exact sequences
$$
0\rightarrow \mathcal O_U(-D_n)\rightarrow \mathcal O_U\rightarrow \mathcal O_{D_n}\rightarrow 0,
$$
yielding, since the dimension of the preimage of $m_R$ in $U$ is two and $R$ is normal,  exact sequences
$$
\begin{array}{l}
0\rightarrow R/I(nF)\rightarrow H^0(D_n,\mathcal O_{D_n})\rightarrow H^1(U,\mathcal O_U(-D_n))\rightarrow H^1(U,\mathcal O_U)\rightarrow H^1(D_n,\mathcal O_{D_n})\\
\rightarrow H^2(U,\mathcal O_U(-D_n))\rightarrow H^2(U,\mathcal O_U)\rightarrow H^2(U,\mathcal O_{D_n})\rightarrow 0
\end{array}
$$
giving the formula
$$
\begin{array}{lll}
\lambda_R(R/I(nF))&=&\chi(\mathcal O_{D_n})-\lambda_RH^1(U,\mathcal O_U(-D_n))+\lambda_RH^1(U,\mathcal O_U)\\
&&+\lambda_RH^2(U,\mathcal O_U(-D_n))-\lambda_RH^2(U,\mathcal O_U).
\end{array}
$$
Formula (\ref{form1}) is now obtained from Lemma \ref{Lemmabound}.

We can compute $\chi(\mathcal O_{D_n})$ by regarding $D_n$ as a closed subscheme of the projective variety $Y$. From the exact sequence
$$
0\rightarrow \mathcal O_Y(-D_n)\rightarrow \mathcal O_Y\rightarrow \mathcal O_{D_n}\rightarrow 0,
$$
we obtain that 
\begin{equation}\label{form2}
\chi(\mathcal O_{D_n})=\chi(\mathcal O_Y)-\chi(\mathcal O_Y(-D_n)).
\end{equation}
Now from the Riemann Roch theorem for three dimensional nonsingular varieties (Exercise A.6.7 
\cite{H}) 
\begin{equation}\label{form3}
\begin{array}{lll}
\chi(\mathcal O_Y(-D_n))&=&\frac{1}{12}(-D_n\cdot (-D_n-K_Y)\cdot(2(-D_n)-K_Y))\\
&&+\frac{1}{12}((-D_n)\cdot c_2(Y))+1-p_a(Y)\\
&=& -\frac{1}{6}(D_n^3)-\frac{1}{4}(D_n^2\cdot K_Y)-\frac{1}{12}(D_n\cdot K_Y^2+c_2(Y))+1-p_a(Y).
\end{array}
\end{equation}
Thus from (\ref{form1}), we obtain that there is a function $\delta(n)$ such that
\begin{equation}\label{form10}
\lambda_R(R/I(nF))=\frac{1}{6}(D_n^3)+\frac{1}{4}(D_n^2\cdot K_Y)+\delta(n)
\end{equation}
where there is a constant $c_1$ such that $|\delta(n)|<c_1n$ for all $n\in \NN$. 

Let $\alpha=\frac{3}{9-\sqrt{3}}$, so that
$$
D_n=\lceil \alpha n\rceil\overline S+nF.
$$
We have
$$
\alpha=\frac{3}{9-\sqrt{3}}=\frac{9}{26}+\frac{\sqrt{3}}{26},\,\, \alpha^2=\frac{84}{26^2}+\frac{18}{26^2}\sqrt{3},\,\,
\alpha^3=\frac{810}{26^3}+\frac{246}{26^3}\sqrt{3}.
$$

On page 16 of \cite{C}, it is shown that
\begin{equation}\label{form4}
(\overline S^3)=468,\,\, (\overline S^2\cdot F)=-162,\,\, (\overline S\cdot F^2)=54,\,\, (F^3)=54.
\end{equation}
We compute that
\begin{equation}\label{form5}
\begin{array}{lll}
(D_n^3)&=&\lceil \alpha n\rceil^3(\overline S^3)+3n\lceil \alpha n\rceil^2(\overline S^2\cdot F)+3n^2\lceil\alpha n\rceil(\overline S\cdot F^2)+n^3(F^3)\\
&=& 468\lceil \alpha n\rceil^3-486n\lceil \alpha n\rceil^2+162 n^2\lceil\alpha n\rceil+54 n^3
\end{array}
\end{equation}
and
\begin{equation}\label{form6}
(D_n^2\cdot K_Y)=\lceil \alpha n\rceil^2(\overline S^2\cdot K_Y)+2n\lceil n\alpha\rceil (\overline S\cdot F\cdot K_Y)+n^2(F^2\cdot K_Y).
\end{equation}

Thus
\begin{equation}\label{form21}
\lim_{n\rightarrow\infty}\frac{(D_n)^3}{n^3}=468 \alpha^3-486\alpha^2+162\alpha+54=\frac{12042}{169}-\frac{27}{169}\sqrt{3}.
\end{equation}

We obtain 
\begin{equation}\label{form7}
e(\mathcal I(F))=3!\left(\lim_{n\rightarrow\infty}\frac{\lambda_R(R/I(nF))}{n^3}\right)=3!(\alpha^3468-486\alpha^2+162\alpha+54)
=\frac{72252}{169}-\frac{162}{169}\sqrt{3}.
\end{equation}

%(corrected from $\frac{313092}{169}-\frac{702}{169}\sqrt{3}$)
%giving the (corrected) result of Theorem 1.4 \cite{C}, which was calculated by a different method.

We give the numeric expression of \ref{form6}, which is computed using the formulas (\ref{form4}) and other formulas from \cite{C}. This formula is not necessary to compute the asymptotic behavior of the first difference function $\lambda_R(I(nF))/I((n+1)F)$.
\begin{equation}\label{form22}
(D_n^2\cdot K_Y)=-792\lceil n\alpha\rceil^2+564n\lceil n\alpha\rceil -175 n^2.
\end{equation}

We now calculate the first difference function.

\begin{equation}\label{form8}
\lambda_R(I(nF)/I((n+1)F))=\lambda_R(R/I((n+1)F))-\lambda_R(R/I(nF))=\frac{1}{6}\left((D_{n+1}^3)-(D_n^3)\right)+\epsilon(n)
\end{equation}
where there is a constant $c''$ such that $|\epsilon(n)|<c''n$ for $n\in\NN$. We have
\begin{equation}\label{form9}
\begin{array}{lll}
\left((D_{n+1}^3)-(D_n^3)\right)&=&
468\left(\lceil\alpha(n+1)\rceil^3-\lceil\alpha n\rceil^3\right)-486\left((n+1)\lceil\alpha(n+1)\rceil^2-n\lceil\alpha n\rceil^2\right)\\
&&+162\left((n+1)^2\lceil\alpha(n+1)\rceil-n^2\lceil\alpha n\rceil\right)+54\left((n+1)^3-n^3\right)\\
&=&468(\sigma_{\alpha}(n)(\lceil\alpha(n+1)\rceil^2+\lceil\alpha(n+1)\rceil\lceil\alpha n\rceil+\lceil\alpha n\rceil^2)\\
&&-486\left(\sigma_{\alpha}(n)(\lceil\alpha(n+1)\rceil n+\lceil\alpha n\rceil n)+\lceil\alpha(n+1)\rceil^2\right)\\
&&+162\left(\sigma_{\alpha}(n) n^2+\lceil\alpha(n+1)\rceil(2n+1)\right)
+54(3n^2+3n+1)
\end{array}
\end{equation}

We have that
$$
\lceil\alpha\rceil=\lceil\frac{3}{9-\sqrt{3}}\rceil=1\mbox{ and }\lfloor \alpha\rfloor = 0.
$$
By Proposition \ref{Prop1*}, the positive integers are the disjoint union $\ZZ_{>0}=\Sigma_1\coprod \Sigma_2$ of two infinite sets, where $\sigma_{\alpha}(n)=0$ if $n\in \Sigma_1$ and $\sigma_{\alpha}(n)=1$ if $n\in \Sigma_2$.

The limit as $n\rightarrow\infty$ restricted to $n\in \Sigma_1$ of $\frac{\lambda_R(I(nF)/I((n+1)F)}{n^2}$ is thus
$$
\frac{1}{6}\left(-486\alpha^2+324\alpha+162\right)
=\frac{144504}{6(26)^2}-\frac{324}{6(26)^2}\sqrt{3}
$$
and the limit as $n\rightarrow\infty$ restricted to $n\in \Sigma_2$ of $\frac{\lambda_R(I(nF)/I((n+1)F)}{n^2}$ is
$$
\frac{1}{6}\left(918\alpha^2-810\alpha+324\right)
=\frac{106596}{6(26)^2}-\frac{4536}{6(26)^2}\sqrt{3}.
$$
These two limits, over $n$ in $\Sigma_1$ and $\Sigma_2$ respectively are not equal, so the limit
$$
\lim_{n\rightarrow\infty}\frac{\lambda_R(I(nF)/I((n+1)F)}{n^2}
$$
does not exist.

	\section{Symbolic powers with infinitely many Rees valuations}\label{SecSym}

In this section we construct an example of a height one prime ideal in a two dimensional normal local ring $R$ such that for all $n\ge 1$, the symbolic power $P^{(n)}$ has a Rees valuation other then the $PR_P$-adic  valuation $v_P$, but for  $1\le m<n$, $v_P$ is the only common Rees valuation of $P^{(m)}$ 
and $P^{(n)}$. In particular, the set of Rees valuations of $P^{(n)}$ over all  $n\ge 1$ is infinite. 

The symbolic algebra $\oplus_{n\ge 0}P^{(n)}$ is not finitely generated, for instance by Corollary 5.2 \cite{CPS}, since
the analytic spread of $P^{(n)}$ is  $2=\dim R$ for all $n$.  

We will show that the analytic spread 
\begin{equation}\label{E5}
\ell(\oplus_{n\ge 0}P^{(n)})=\dim (\oplus_{n\ge 0}P^{(n)})/m_R(\oplus_{n\ge 0}P^{(n)})=0,
\end{equation}
 where $m_R$ is the maximal ideal of $R$. Thus the fiber $\psi^{-1}(m_R)=\emptyset$, where 
 $$
 \psi:\mbox{Proj}(\oplus_{n\ge 0}P^{(n)})\rightarrow\mbox{Spec}(R)
 $$
  is the natural projection.

We now make the construction. Let $E$ be an elliptic curve over an algebraically closed field $k$ and $p\in E$ be  a (closed) point. Let $\mathcal L=\mathcal O_E(p)$ and $V(\mathcal L)=\mbox{Spec}(\oplus_{n\ge 0}\mathcal L^n)$, the total space of the line bundle $\mathcal L^{-1}$, with canonical morphism $\pi:V(\mathcal L)\rightarrow E$. By Propositions II.8.4.2 and II.4.4 \cite{EGA}, we have a natural inclusion 
$$
V(\mathcal L)\subset \PP(\mathcal L\oplus\mathcal O_E)\cong \PP(\mathcal O_E\oplus\mathcal L^{-1})
$$
 and the zero section $E_0$ of $\pi$ is the image of the inclusion $\PP(\mathcal L^{-1})\subset \PP(\mathcal O_E\oplus \mathcal L^{-1})$. By Proposition V.2.9 \cite{H}, 
 $$
 \mathcal O_{V(\mathcal L)}(E_0)\otimes\mathcal O_{E_0}\cong \mathcal O_{\PP(\mathcal O_E\oplus\mathcal L^{-1})}(E_0)\otimes\mathcal O_{E_0}\cong \mathcal O_E(-p).
 $$
  Let $A=\oplus_{n\ge 0}\Gamma(E,\mathcal L^n)$.
  $A$ is a normal domain. 
By Proposition II.8.8.2 and Remark II.8.8.3 \cite{EGA}, 
there exists a canonical  projective morphism $V(\mathcal L)\rightarrow \mbox{Spec}(A)$ which contracts the zero section of $E_0$ of $V(\mathcal L)$ to the graded maximal ideal $\mathfrak m$ of $A$ and is an isomorphism on $V(\mathcal L)\setminus E_0\rightarrow  \mbox{Spec}(A)\setminus\{\mathfrak m\}$. Let $R=\hat A$, the completion of $A$ at $\mathfrak m$. $R$ is a complete local domain. Let $m_R$ be the maximal ideal of $R$. Let $X=V(\mathcal L)\times_{\mbox{Spec}(A)}\mbox{Spec}(R)$, with natural birational projective morphism $\phi:X\rightarrow \mbox{Spec}(R)$.
We will identify $E_0$ with $E$. 

There exists a point $q\in E$ such that  $\mathcal O_E(q-p)$ has infinite order in the Jacobian of $E$. Since $R$ is complete, there exists a curve $C$ on the nonsingular scheme $X$ such that $C\cdot E=q$. Let $P=\Gamma(X,\mathcal O_X(-C))$, a height one prime ideal in $R$. The $n$-th symbolic power of $P$ is $P^{(n)}=\Gamma(X,\mathcal O_X(-nC))$. The Zariski decomposition of $C$ is $\Delta=C+E$ (Lemma 4.1 \cite{C1}) and $\Gamma(X,\mathcal O_X(-nC))=\Gamma(X,\mathcal O_X(-n\Delta))$ for all $n>0$ by Lemma 4.3 \cite{C1}. By Lemma 2.2  \cite{C1}, for all $n>0$, the base locus of $-n\Delta$ is a subset of $E$. The curve $E$ is in the base locus of $-n\Delta$ for all $n>0$ since $\mathcal O_X(-n\Delta)\otimes\mathcal O_{E}\cong \mathcal O_E(n(q-p))$ and so $\Gamma(E,\mathcal O_X(-n\Delta)\otimes\mathcal O_E)=0$. Thus 
\begin{equation}\label{E6}
\Gamma(X,\mathcal O_X(-nC))=\Gamma(X,\mathcal O_X(-n\Delta-E))
\end{equation}
 for all $n\ge 1$.

We have short exact sequences 
\begin{equation}\label{E1}
0\rightarrow \mathcal O_X(-n\Delta-2E)\rightarrow \mathcal O_X(-n\Delta -E)\rightarrow \mathcal O_X(-n\Delta-E)\otimes\mathcal O_{E}\rightarrow 0.
\end{equation}
By Theorem III.11.1 \cite{H},
$$
H^1(X,\mathcal O_X(-n\Delta-2E))\cong \lim_{\leftarrow}H^1(X,\mathcal O_X(-n\Delta-2E)\otimes\mathcal O_{mE}),
$$
where the limit is over the schemes $mE$. By Serre duality on an elliptic curve, $H^1(E,\mathcal O_E(D))=0$ if $D$ is a divisor on $E$ of positive degree. From induction on $m$ in the short exact sequences
$$
0\rightarrow \mathcal O_X(-n\Delta-2E-mE)\otimes \mathcal O_E\rightarrow \mathcal O_X(-n\Delta-2E)\otimes\mathcal O_{(m+1)E}\rightarrow \mathcal O_X(-n\Delta-2E)\otimes\mathcal O_{mE}\rightarrow 0,
$$
we have that 
$$
H^1(X,\mathcal O_X(-n\Delta-2E)\otimes \mathcal O_{(m+1)E})\cong H^1(X,\mathcal O_X(-n\Delta-2E)\otimes \mathcal O_{mE})
$$
for all $m\ge 1$. Since $H^1(X,\mathcal O_X(-n\Delta-2E)\otimes\mathcal O_E)=0$, we have that 
$$
H^1(X,\mathcal O_X(-n\Delta-2E)\otimes\mathcal O_{mE})=0
$$
 for all $m\ge 1$. Thus
\begin{equation}\label{E2}
H^1(X,\mathcal O_X(-n\Delta-2E))=0
\end{equation}
 for all $n$. By the Riemann Roch theorem on $E$, $h^0(\mathcal O_E(n(q-p)+p))=1$ for all $n$. Thus there exists a unique point $q_n\in E$ such that $\mathcal O_E(n(q-p)+p)\cong \mathcal O_E(q_n)$. Since $q-p$ has infinite order in the Jacobian of $E$, we have that the $q_n$ are distinct points 
 and $q_n\ne q$ for all $n$. By (\ref{E1}) and (\ref{E2}), we have surjections
 $$
 \Gamma(X,\mathcal O_X(-n\Delta-E))\rightarrow \Gamma(E,\mathcal O_X(-n\Delta-E)\otimes\mathcal O_E)=\Gamma(E,\mathcal O_E(q_n)).
 $$
Thus the base locus of $\Gamma(X,\mathcal O_X(-n\Delta-E))$ is the point $q_n$, and so for all $n$, 
\begin{equation}\label{E3}
P^{(n)}\mathcal O_X=\mathcal O_X(-n\Delta-E)\mathcal I_{n}
\end{equation} 
 where $\mathcal I_{n}$ is an ideal sheaf on $X$ such that the support of $\mathcal O_X/\mathcal I_n$ is $q_n$.

Let $\gamma_n:B_n\rightarrow \mbox{Spec}(R)$ be the normalization of the blow up of $P^{(n)}$. The Rees valuations of $P^{(n)}$ whose center on $R$ is $m_R$ are the valuations associated to the local rings $\mathcal O_{B_n,G}$ where $G$ is a prime exceptional divisor of $\gamma_n$. Let $C_n$ be the normalization of the blowup of $P^{(n)}\mathcal O_X$. We have  birational $\mbox{Spec}(R)$ morphisms
$$
\begin{array}{ccccc}
&&C_n&&\\
&\swarrow&&\searrow&\\
B_n&&&&X.\\
\end{array}
$$
The morphism $C_n\rightarrow X$ is not an isomorphism and there does not exist an $R$-morphism $X\rightarrow B_n$ since $P^{(n)}\mathcal O_X$ is not invertible. By Zariski's main theorem, there exists a curve $F$ on $C_n$ which maps to the point $q_n$ on $X$ and maps to a curve on $B_n$. Thus the valuation associated to the local ring $\mathcal O_{C_n,F}$ is a Rees valuation of $P^{(n)}$. 
Since the points $q_n$ are all distinct, we conclude that the ideals $P^{(n)}$ have infinitely many Rees valuations. 

It remains to show that if $m\ne n$, then there are  no common Rees valuations of  $P^{(m)}$ and $P^{(n)}$  which dominate the maximal ideal $m_R$ of $R$. By the above analysis,  any exceptional curve for $B_n\rightarrow \mbox{Spec}(R)$ either comes from a curve on $C_n$ which contracts to $q_n$ on $X$ or is the image of the strict transform of $E$ on $C_n$.
 Since the $q_n$ are all distinct, we are reduced to showing that the valuation of $\mathcal O_{X,E}$ is not a Rees valuation of $P^{(n)}$ for any $n$.

Let $\alpha_n:Y_n\rightarrow X$ be the blow up of $q_n$. Let $C_n$ be the strict transform of $C$ on $Y_n$, $E_n$ be the strict transform of $E$ and let $F_n$ be the prime exceptional divisor of $Y_n\rightarrow X$. The curve $E_n$ is isomorphic to $E$. We have the intersection products
$E_n\cdot C_n=q$ and  $F_n\cdot E_n=q_n$. Further, 
$$
-p=E\cdot E=E_n\cdot \alpha_n^*E=E_n\cdot(E_n+F_n)=E_n\cdot E_n+q_n.
$$
Thus $E_n\cdot E_n=-p-q_n$, and so 
\begin{equation}\label{E4}
(\alpha_n^*(-n\Delta-E)-F_n)\cdot E_n=q_n-q_n=0.
\end{equation}
 Since $\mathcal O_X/\mathcal I_n$ is supported at $q_n$, $\mathcal I_n\mathcal O_{Y_n}=\mathcal O_{Y_n}(-F_n)\mathcal J_n$ where $\mathcal J_n$ is an ideal sheaf on $Y_n$ such that $\mathcal O_{Y_n}\mathcal J_n$ is supported on $F_n$. By (\ref{E3}), we have that 
$$
P^{(n)}\mathcal O_{Y_n}=\mathcal O_{Y_n}(\alpha_n^*(-n\Delta-E)-F_n)\mathcal J_n.
$$
Now 
$$
\begin{array}{lll}
P^{(n)}&=&\Gamma(Y_n,\mathcal O_{Y_n}(\alpha_n^*(-n\Delta-E)-F_n)\mathcal J_n)
\subset \Gamma(Y_n,\mathcal O_{Y_n}(\alpha_n^*(-n\Delta-E)-F_n))\\
&\subset &\Gamma(Y_n,\mathcal O_{Y_n}(-n\overline C))=P^{(n)},
\end{array}
$$
 so
$$
P^{(n)}=\Gamma(Y_n,\mathcal O_{Y_n}(\alpha_n^*(-n\Delta-E)-F_n)).
$$
Now $E_n$ is not in the base locus of $\alpha_n^*(-n\Delta-E)-F_n$ by (\ref{E3}), so the image of 
$\Gamma(Y_n,\mathcal O_{Y_n}(\alpha_n^*(-n\Delta-E)-F_n))$ in  $\Gamma(E_n,\mathcal O_{Y_n}(-n\Delta-E)-F_n)\otimes \mathcal O_{E_n})$ is nonzero.  Further,
$\mathcal O_{Y_n}(\alpha_n^*(-n\Delta-E)-F_n)\otimes \mathcal O_{E_n}\cong \mathcal O_E$ by (\ref{E4}). Thus the image of $\Gamma(Y_n,\mathcal O_{Y_n}(\alpha_n^*(-n\Delta-E)-F_n))$  is equal to $\Gamma(E_n,\mathcal O_{E_n})=k$. Thus $E_n$ is disjoint from the base locus of $\alpha_n^*(-n\Delta-E)-F_n$, and thus $\mathcal O_{Y_n}/\mathcal J_n$ is supported at a finite number of points which are disjoint from $E_n$. 

By continuing to blow up points above the support of $\mathcal O_{Y_n}/\mathcal J_n$, we construct  a morphism $\beta_n:X_n\rightarrow Y_n$ such that $P^{(n)}\mathcal O_{X_n}$ is invertible. We then have that
$$
P^{(n)}\mathcal O_{X_n}=\mathcal O_{X_n}(-n\overline C_n-(n+1)\overline E_n-(n+2)\overline F_n-\sum a_iG_i)
$$
where $\overline C_n, \overline E_n,\overline F_n$ are the respective strict transforms of $C_n,E_n$ and $F_n$, $a_i$ are positive integers, and the $G_i$ are exceptional for $\beta_n$ and thus are disjoint from $\overline E_n$. We have that
$$
E_n\cdot (-n\overline C_n-(n+1)\overline E_n-(n+2)\overline F_n-\sum a_iG_i)=0
$$
so the valuation associated to $\mathcal O_{X,E}=\mathcal O_{X_n,\overline E_n}$ is not a Rees valuation  of $P^{(n)}$ by Corollary 6.4 \cite{C1}. 

Finally, we prove equation (\ref{E5}). Let $Q=\oplus_{n\ge 0} \Gamma(X,\mathcal O_X(-n\Delta-E))$. $Q$ is a prime ideal in 
$\oplus_{n\ge 0}P^{(n)}=\oplus_{n\ge 0}\Gamma(X,\mathcal O_X(-n\Delta))$ (as explained before Remark 6.6 \cite{C1}) and since $E$ is the only exceptional curve of $\phi:X\rightarrow \mbox{Spec}(R)$, $Q=\sqrt{m_R(\oplus_{n\ge 0}P^{(n)})}$ by Proposition 6.7 \cite{C1}. 

By Equation (\ref{E6}) and since $\Gamma(X,\mathcal O_X(-E))=\{f\in R\mid v_E(f)>0\}=m_R$, where $v_E$ is the valuation associated to $E$,
$Q=m_R\oplus(\oplus_{n>0}P^{(n)})$. Thus 
$$
\ell(\oplus_{n\ge 0}P^{(n)})=\dim \left[(\oplus_{n\ge 0}P^{(n)})/m_R(\oplus_{n\ge 0}P^{(n)})\right]
=\dim\left[ (\oplus_{n\ge 0}P^{(n)})/Q\right]=\dim R/m_R=0.
$$

\end{document}